\documentclass[11pt]{article}
\usepackage{fancyhdr}

\usepackage{amsmath}
\usepackage{mathtools}
\usepackage{amsfonts}
\usepackage{amsthm}
\usepackage{amssymb}
\usepackage{color}
\usepackage{dsfont}
\usepackage{geometry} 
 \newgeometry{tmargin=2.5cm, bmargin=3.6cm, lmargin=2.2cm, rmargin=2.2cm} 

\usepackage{color}

\title{Gradient estimates for problems with Orlicz growth}

\usepackage{authblk}
 
\author[$\diamond$]{Iwona Chlebicka\thanks{e-mail address: iskrzypczak@mimuw.edu.pl\\The research is supported by NCN grant no. 2016/23/D/ST1/01072.}} 
\affil[$\diamond$]{\small
Institute of Mathematics, Polish Academy of Sciences, \newline
ul. \'{S}niadeckich 8, 00-656 Warsaw, Poland\newline

Faculty of Mathematics, Informatics and Mechanics, University of Warsaw, \newline
ul. Banacha 2, 02-097 Warsaw, Poland
} 

\date{}
\begin{document}
\maketitle \sloppy

\thispagestyle{empty}

\renewcommand{\it}{\sl}
\renewcommand{\em}{\sl}

\belowdisplayskip=18pt plus 6pt minus 12pt \abovedisplayskip=18pt
plus 6pt minus 12pt
\parskip 4pt plus 1pt
\parindent 0pt

\newcommand{\barint}{
         \rule[.036in]{.12in}{.009in}\kern-.16in
          \displaystyle\int  } 
\def\R{{\mathbb{R}}}
\def\r{{\mathbb{R}}}
\def\n{{\mathbb{N}}}
\def\bu{{\bar{u}}}
\def\bg{{\bar{g}}}
\def\bG{{\bar{G}}}
\def\ba{{\bar{a}}}
\def\bv{{\bar{v}}}
\def\bmu{{\bar{\mu}}}
\def\rn{{\mathbb{R}^{n}}}
\def\rN{{\mathbb{R}^{N}}}
\def\avenorm#1{\mathchoice%
          {\mathop{\kern 0.2em\vrule width 0.6em height 0.69678ex depth -0.58065ex
                  \kern -0.545em \|{#1}\|}}%
          {\mathop{\kern 0.1em\vrule width 0.5em height 0.69678ex depth -0.60387ex
                  \kern -0.495em \|{#1}\|}}%
          {\mathop{\kern 0.1em\vrule width 0.5em height 0.69678ex depth -0.60387ex
                  \kern -0.495em \|{#1}\|}}%
          {\mathop{\kern 0.1em\vrule width 0.5em height 0.69678ex depth -0.60387ex
                  \kern -0.495em \|{#1}\|}}}

\newtheorem{theo}{\bf Theorem} 
\newtheorem{coro}{\bf Corollary}[section]
\newtheorem{lem}{\bf Lemma}[section]
\newtheorem{rem}{\bf Remark}[section]
\newtheorem{defi}{\bf Definition}[section]
\newtheorem{ex}{\bf Example}[section]
\newtheorem{fact}{\bf Fact}[section]
\newtheorem{prop}{\bf Proposition}[section]

\newcommand{\dv}{{\rm div}}
\newcommand{\wt}{\widetilde}
\newcommand{\ve}{\varepsilon}
\newcommand{\vp}{\varphi}
\newcommand{\vt}{\vartheta}
\newcommand{\gb}{{g_\bullet}}
\newcommand{\gbn}{{(\gb)_n}}
\newcommand{\vr}{\varrho}
\newcommand{\pa}{\partial}
\newcommand{\MDu}{{M^*_{0;2B_0}(|Du|)}}
\newcommand{\Mmu}{{M^*_{1;2B_0}(\mu)}}
\newcommand{\Mb}{{M^*_{0;2B}}}
\newcommand{\Mbm}{{M^*_{1;2B}}}


\parindent 1em

\begin{abstract}

We study a general nonlinear elliptic equation in the Orlicz setting with data not belonging to the dual of the energy space. We provide several Lorentz-type and Morrey-type estimates for the gradients of solutions under various conditions on the data.

\end{abstract}

\bigskip
\bigskip


\bigskip
\bigskip


\bigskip

\section{Introduction}

Studying the elliptic equations with data on the right-hand side less regular than naturally belonging to the dual space of the leading part of the operator we fall out of the scope of the classical regularity theory and the derivation of estimates on solutions and their gradients becomes challenging. 

\smallskip

We concetrate on the problems with the leading part of the operator in the Orlicz class and provide estimates on their solutions in the scales of rearrangement invariant Lorentz-type, as well as not rearrangement invariant Morrey-type spaces, depending on the type of data. The setting is as follows.
 We investigate regularity of the solutions obtained as a limit of approximation, SOLA for short, to a~natural generalisation of $p$-Laplace problems, that is to the nonlinear elliptic equation \[-\dv\,a(x,Du)=\mu\]
 described precisely in Assumption (A) in Section~\ref{sec:results}. Let us stress here that the dependence on $x$ of $a(x,\cdot)$ is just measurable. The model example that we admitt, which comes from~\cite{Lieb91,Lieb93} and is studied e.g. in~\cite{Baroni-Riesz,DieEtt}, reads as follows
\begin{equation}
\label{intro:eq:main} -{ \dv}\left(\frac{\omega(x)g(|D u|)}{|D u|}  D u \right)=\mu\quad\text{in}\quad \Omega,
\end{equation}
 where $\Omega$ is a bounded domain in $\rn$, $n\geq 2$, $\omega:\Omega\to [c,\infty)$ is a bounded measurable and separated from zero function, $\mu$ is  a~Borel measure with finite total mass, $|\mu|(\Omega)<\infty$, while $g\in C^1(0,\infty)$ is a nonnegative, increasing, and convex function, such that
\begin{equation}\label{ig-sg}  1\leq i_g=\inf_{t>0}\frac{tg'(t)}{g(t)}\leq \sup_{t>0}\frac{tg'(t)}{g(t)}=s_g<\infty
,
\end{equation}
which in particular implies $g\in\Delta_2$.

The gradient estimates for SOLA to the problem~\eqref{intro:eq:main} relates to the regularity of local minimisers of~variational problem
\[{\cal G}(Dw)=\int_\Omega G(|D w|)\,dx\]
with $G$ being a primitive of $g$ from~\eqref{intro:eq:main}, for which $
2\leq 1+i_g= i_G\leq
s_G=1+s_g
.$

This study includes the classical case of $p$-Laplace equation\[-\Delta_p u =\mu,\] where $G(t)=G_p(t)=|t|^p$ and $i_G=s_G=p$, retrieving certain already classical results mentioned below. Note that we admit the corresponding equation with measurable coefficients. Other examples of admissible modular functions are e.g. $G(t)=G_{p,\alpha}(t)=|t|^p\log^\alpha(e+|t|)$, their multiplications and compositions.

\subsubsection*{State of the art} Regularity of solutions to elliptic differential equations of the form 
\begin{equation}
\label{intro:eq:plap}-\Delta_p u= f \in L^q
\end{equation}
is a well understood topic. The cases of quickly and slowly growing operators are highly different. Indeed, only if $p>n$ or \[\text{if}\qquad p\leq n \ \text{ and }\ q>\frac{np}{np-n+p}=(p^*)',\qquad \text{ then }\qquad f\in L^q\subset W^{-1,p'}=(W^{1,p})^*\] and~\eqref{intro:eq:plap}  can be uniquely solved in the natural energy space, which is covered by the classical regularity theory, e.g.~\cite{DiBMan,Iw}. Therefore, the effort concentrates on the more demanding case of slowly growing operators related to $p\leq n$ and small $q$. Note that then the notion of weak solution is  too restrictive to consider $L^q$-data, whereas the distributional solutions can be not unique, cf. the classical linear example~\cite{Serrin-pat}. Therefore, a special notion of~solutions keeping uniqueness has to be introduced. We consider SOLA defined in Section~\ref{ssec:SOLA}, where the other relevant notions, i.e. entropy and renormalized solutions, are also commented. Let us mention that all the results involving SOLA naturally concerns only $p>2-1/n$, since it is necessary to ensure that $u\in W^{1,1}_{loc}(\Omega)$ for arbitrary data.

Let us point out that the most subtle case, that is the conformal case of $p=n$, will not be an objective for our study and we only refer to~\cite{BDSZ,DHM2, DHM}. Further we concentrate on $p< n$.

There are known  estimates on gradients of solutions to~\eqref{intro:eq:plap} in the scales of the Lebesgue, Lorentz-type and Morrey-type spaces, depending on the type of data.  Our already classical inspiration was the result obtained by Boccardo and Galou\"et~\cite{bgSOLA-cpde} for solutions to~\eqref{intro:eq:plap} that yields
\begin{equation}
\label{BG-class}
f\in L^q\implies|Du|^{p-1}\in L^{\frac{nq}{n-q}}\quad\text{ when }\quad 2-\frac{1}n<p<n \quad\text{ and }\quad 1<q<\frac{np}{np-n+p},
\end{equation}
where the range is sharp in the scale of the Lebesgue spaces. We shall  only mention the remarkable results for the systems~\cite{DHM-p} and go back to equation~\eqref{intro:eq:plap}. The above mentioned result by Boccardo and Galou\"et was upgraded by Mingione in~\cite{min-grad-est} to cover very general rearrangement invariant function spaces. For instance, in~\cite{min-grad-est} the provided regularity is formulated in a weak-type setting of Lorentz and Marcinkiewicz spaces, i.e.
\begin{equation}
\label{Mingione-Lor}
f\in L(q,s)\implies|Du|^{p-1}\in L\left({\frac{nq}{n-q}},s\right)\ \text{ when }\ s\in(0,\infty],\ 2\leq p<n,  \text{ and }  1<q\leq\frac{np}{np-n+p}.
\end{equation}
See also~\cite{ku-min-univ, KuMi} for related estimates in the case of $2-1/n<p<2$.

On the other hand, in the study of solutions to the generalization of~\eqref{intro:eq:plap} with the Morrey data, i.e. satisfying the following density condition\[f\in L^{q,\theta}\qquad\text{if}\qquad\int_{ B_R }|f|^q\,dx\leq c R^{n-\theta}\qquad\text{with}\qquad  \theta\leq n\]
(cf. Definition~\ref{def:Mor:sp}), Mingione in~\cite{min-grad-est} states that
\begin{equation}
\label{SAM-class}
f\in L^{q,\theta}\implies|Du|^{p-1}\in L^{\frac{\theta q}{\theta-q},\theta}\quad\text{ for }\quad 2 \leq p<\theta \quad\text{ and }\quad 1<q\leq\frac{\theta p}{\theta p-\theta+p}.
\end{equation}
Note that in comparison to~\eqref{Mingione-Lor} describing the range of parameters we change $n$ to $\theta$. See the classical paper by Stampacchia~\cite{stamp} for the first proof of the linear case ($p=2$) of~\eqref{SAM-class}, \cite{adams-note} by Adams for the sharp linear version, and~\cite{min-grad-est} by Mingione for the nonlinear one ($p\geq 2$), its sharpness, and the corresponding result in the Lorentz-Morrey setting
\[f\in L^{\theta}(q,s) \implies  |Du|^{p-1}\in L^\theta\left(\frac{n q}{n-q},\frac{ns}{n-q}\right)\]
within the whole above range of parameters and including all bordeline cases.

Among vastness of gradient integrability results derived for solutions to problems with power-type growth corresponding to~\eqref{intro:eq:plap} we would like to mention e.g. those obtained in the Lorentz spaces for entropy solutions~\cite{KiLi}, in the Marcinkiewicz spaces for SOLA~\cite{BocMarcSola}, and for SOLA in multiple scales of the Lorentz and the Morrey-type in the comprehensive paper~\cite{min-grad-est}. Corresponding results for systems are provided in~\cite{KuMi-vec-npt}. Let us refer also to some other attempts~\cite{AdPh,AlCiSb1,CiMa14,BGM,DMM,DHM2,DHM,MPhuc,min07,Phuc}.

We stress that the issue of gradient estimates for $L^1$ or measure data is deeply investigated in the Sobolev setting, but little is known in the Orlicz spaces, where we want to contribute. To our best knowledge in the Orlicz setting the Marcinkiewicz estimates are restricted to~\cite{ACCZG,CGZG,CiMa}, while the Lorentz or the Morrey estimates for problems posed in the Orlicz spaces are not known yet at all.

\subsubsection*{Outline of the results  } Our aim is to provide precise analysis of the case related to $2\leq p<n$ in~\eqref{intro:eq:plap}. As expected, the role of~$p$ in the range bounds~\eqref{Mingione-Lor} and~\eqref{SAM-class} is played in the corresponding results by $i_G$ or $s_G$, whereas the role of $t\mapsto |t|^{p-1}$ is taken by $t\mapsto g(t)$. 

 The precise formulations are presented and commented  in Section~\ref{sec:results}. Let us mention here only the key accomplishments. We prove the estimates for gradients of solutions to~\eqref{intro:eq:main} having natural form taking into account various functional settings of data. All necessary definitions are given in Section~\ref{ssec:fn-sp} in Appendix.

Theorem~\ref{theo:Lorentz-est} provides the following result in the scale of the Lorentz spaces
\[\mu\in L(q,s)\ \ \implies\ \ g(|Du|)\in L\left(\frac{nq}{n-q},s\right)\quad\text{for}\quad 1<q\leq \frac{n i_G}{ns_G-n+i_G} \ \text{ and }\  s\in(0,\infty)\]
covering also the Marcinkiewicz case
\[\mu\in {\cal M}^q\quad\implies\quad g(|Du|)\in  {\cal M}^{\frac{nq}{n-q}}\qquad\text{for}\qquad 1<q\leq \frac{n i_G}{ns_G-n+i_G}.\]
Since the $p$-Laplace case concerns $i_G=s_G=p$ and $g(t)=|t|^{p-1}$,  the above results for Lebesgue's data ($q=s$) imply~\eqref{BG-class}. Moreover, for $g(t)=|t|^{p-1}\log^\beta(1+|t|)$, we recover \cite[Example~3.4]{CiMa} within the prescribed scope of parameters.

 Due to Corollary~\ref{coro:LLogL-est}, for solutions to $L\log L$-data problem we get 
\[\mu\in L\log L\quad\implies\quad g^\frac{n}{n-1}(|Du|)\in L^1,\]
relating to the Marcinkiewicz-Orlicz regularity obtained by Cianchi and Mazy'a~\cite{CiMa} for a sligthly different notion of solutions, see Remark~\ref{rem:CiMa}.

Meanwhile, Theorem~\ref{theo:Morrey-est} yields the Morrey gradient regularity (Definition~\ref{def:Mor:sp}) for the problems, where the data are measures  described by the density. Namely, 
\[\mu\in L^{ q,\theta }\ \ \implies\ \ g(|Du|)\in L^{\frac{\theta q}{\theta-q},\theta}\qquad\text{for}\qquad i_G\leq \theta\leq n \quad \text{ and }\quad 1<q\leq \frac{\theta i_G}{\theta s_G-\theta+i_G}
\]
which again in the $p$-Laplace case simplifies to~\eqref{SAM-class}. See that $\theta=n$ is included. Due to Corollary~\ref{coro:Bord-Morrey-est}, for solutions to $L\log L^\theta$-data problem where $2\leq\theta\leq n$, we get 
\[\mu\in L\log L^\theta\quad\implies\quad g^\frac{\theta}{\theta-1}(|Du|)\in L^1.\]
Moreover, within the same range of $\theta,q$ as in the Morrey case above and for $s\in(0,\infty]$ (including the Marcinkiewicz case $s=\infty$), Theorem~\ref{theo:Lor-Mor} provides the most general result in the scale of~the Lorentz-Morrey spaces (Definition~\ref{def:LorMor:sp}), namely
\[\mu\in L^{\theta}(q,s)\quad\implies\quad g(|Du|)\in   L^\theta\left(\frac{\theta q}{\theta-q},\frac{\theta s}{\theta-q}\right).\]
Since it captures the upper bound $q\leq {\theta i_G}/({\theta s_G-\theta+i_G})$, as well as the Marcinkiewicz case  $s=\infty$, it is the Orlicz extension of the main result of \cite[Theorem~11]{min-grad-est} by Mingione.

\subsubsection*{Methods and challenges} 

 The general approach of the paper is careful development of the methods introduced by Mingione in~\cite{min-grad-est} in the generalized $p$-Laplace case. Adapting the framework to the Orlicz setting requires to understand deeply the nature of needed tools and their careful derivation. Note that this approach is rearrangement-free and   enables to study in the same reasoning Lorentz-type and Morrey-type data. Our precise assumptions on the equation we study  are collected in the beginning of  Section~\ref{sec:results}. 
 
We provide estimates for the solution to our main problem~\eqref{intro:eq:main} expressed in the terms of the maximal operator of the data.  Let us indicate that the seminal idea of application of maximal operator to the nonlinear degenerate problems goes  back to~\cite{Iw} and the fundamental results for the nonlinear potential theory are provided by~\cite{KiMa}. The complete theory for equations with $p$-growth is presented in~\cite{KuMi}. We want to stress that the key potential estimates in the Orlicz setting are proven in~\cite{Baroni-Riesz} and they can be directly used to obtain qualitatively the same results as we provide here including also the scope of energy solutions. However,~\cite{Baroni-Riesz} considers only the case of~homogeneous vector field (i.e. $a:\rn\to\rn$), while in our approach $a:\Omega\times\rn\to\rn$ and the dependence $x\mapsto a(x,\cdot)$ is measurable.

 The main tool we derive and apply is the super-level set decay estimate for the maximal operator of~gradient of~solutions, which -- for presenting the intuition -- can be shortened to
\[\begin{split}
&\left|\left\{ M_0(|Du|)>K\lambda\right|\right\}  \leq\frac{1}{G^{\chi}(K)}\left|\left\{ M_0(|Du|)> \lambda\right\}\right|+\left|\left\{  {M_1(|\mu |)}> g(\ve\lambda)\right\}\right|,\end{split}\]
cf. the definition of maximal operators in~\eqref{MDu-Mmu-def} and the full estimate in~\eqref{eq:super-level-est}. Recall again that in the $p$-Laplace case we would have here $G(t)=|t|^p$ and $g(t)=|t|^{p-1}$. Roughly speaking the above inequality settles strongly nonlinear version of the estimate on the level sets of the Hardy-Littlewood maximal function of $|Du|$ by the level sets of the Riesz potential of the data.  See~Proposition~\ref{prop:super-level-est} for this part. What remains afterwards  is proper understanding what information carries the Riesz potential for various types of data. 

Find in Section~\ref{ssec:comments} comments on capturing the upper bound of $q$ in the proofs, which is sharp in the power case.

\subsubsection*{Organization of the paper}

We start the paper giving in Section~\ref{sec:results} the complete set of assumptions and collection of the main results accompanied by comments on their meaning. Afterwards, in Section~\ref{sec:prelim}, we provide Preliminaries that introduce  notation, the Orlicz setting in which our main equation~\eqref{eq:main} is formulated, and information on the notion of solutions we investigate. Section~\ref{sec:aux} is devoted to estimates on the solutions to the homogeneous problem, comparison estimates, and their direct consequences. In Section~\ref{sec:prel-est} we derive our main tool, i.e. super-level set estimates for the maximal operator of the gradient. The final proofs of the main theorems listed above are given in Section~\ref{sec:proofs}. In the end, in Appendix, we give necessary definitions, concise information of~the involved functional spaces, and  basic and classical estimates.

\section{The results}\label{sec:results}

The problem we consider generalizes the problem~\eqref{intro:eq:plap} towards nonstandard growth described in the Orlicz setting. Let us collect here the set of assumptions and present below the main theorems. 

\subsection*{Assumption (A)}

Recall $\Omega$ is a bounded domain in $\rn$, $n\geq 2$. We investigate the problem 
\begin{equation}\label{eq:main}
-{\dv}\, a(x,Du)=\mu\quad\text{in}\quad \Omega,
\end{equation}
where $\Omega\subset\rn$, $\mu$ is  a Borel measure with finite total mass $|\mu|(\Omega)<\infty$, whereas $a:\Omega\times\rn\to\rn$ is a Carath\'{e}odory function (measurable with respect to the first variable and continuous with respect to the second one) that satisfies
\begin{equation*}
\left\{\begin{array}{l}
\langle\partial_z a(x,z)\lambda,\lambda\rangle\geq \nu\frac{g(|z|)}{|z|}\lambda^2,\\
|a(x,z)|+|\partial a(x,z)|\cdot|z|\leq Lg(|z|),
\end{array}\right.
\end{equation*} 
with some increasing and convex function $g\in C^1(0,\infty)$ satisfying~\eqref{ig-sg}, i.e.  \[1\leq i_g=\inf_{t>0}\frac{tg'(t)}{g(t)}\leq \sup_{t>0}\frac{tg'(t)}{g(t)}=s_g<\infty
,\] equivalent to $g,\wt{g}\in\Delta_2$ where $\wt{g}(s)=\sup_{t>0}(t\cdot s-g(t))$ is the Young conjugate of $g$. The primitive of~$g$ is $G\in C^2(0,\infty)$, i.e. $G'(t)=g(t)$. The parameters $i_G=i_g+1$ and $s_G=s_g+1$ describe the speed of growth of $G$.

\subsection*{Main results}

We prove estimates of gradient integrability in several function spaces. See Subsection~\ref{ssec:fn-sp} (in Appendix) for necessary definitions and notation e.g. of the averaged norms. The main proofs are provided in Section~\ref{ssec:main-proofs} with added comments on the range of parameters in Section~\ref{ssec:comments}. Let us present the precise formulations of our main goals.

\medskip

The first result we provide is corresponding to~\eqref{BG-class} and~\eqref{Mingione-Lor}.
\begin{theo}[Estimates in Lorentz spaces]\label{theo:Lorentz-est}
 Let $u\in W^{1,1}_{loc}(\Omega)$ be a local SOLA to the equation~\eqref{eq:main} with $G,g$ satisfying (A) and 
 \begin{equation*}
 1< q \leq \frac{n i_G}{ns_G-n+i_G} \qquad\text{and}\qquad 0<s\leq \infty.
 \end{equation*}
If $\mu\in L(q,s)$ locally in $\Omega,$ then 
 \[ g(|Du|)\in L\left(\frac{nq}{n-q},s\right)\quad\text{locally in }\ \Omega.\]
 
 Moreover, there exists $c=c(n,i_G, s_G,q,s)$, such that for every ball $B_R\subset\subset\Omega$ we have
 \begin{equation}
\label{eq:Lorentz-est}  \avenorm{ g( |Du|) }_{L\left(\frac{nq}{n-q},s\right)(B_{R/2})}\leq c\,g\left( \barint_{B_{2R} }|Du|dx\right) +c \avenorm{\mu}_{L(q ,s )(B_R)}.
\end{equation} 
\end{theo}

In the case of $a:\Omega\times\rn\to\rn$ having the growth comparable to $p$-Laplacian's ($i_G=s_G=p$) we recover the result of \cite[Theorem~13]{min-grad-est}.
\begin{coro}[$p$-Laplace case] 
 Let $u\in W^{1,1}_{loc}(\Omega)$ be a local SOLA to the equation $-\Delta_p u=\mu$ with $p\geq 2$ and 
 \[
 1< q  \leq \frac{n p }{np-n+p }  \qquad\text{and}\qquad 0<s\leq \infty.
 \]
If $\mu\in L(q,s)$ locally in $\Omega$, then
 \[ |Du|^{p-1}\in L\left(\frac{nq}{n-q},s\right)\quad\text{locally in }\ \Omega.\]
Moreover, there exists $c=c(n,p,q,s)$, such that for every ball $B_R\subset\subset\Omega$ we have
 \begin{equation*}
  \avenorm{ | Du |^{p-1}  }_{L\left(\frac{nq}{n-q} ,s \right)(B_{R/2})}\leq c\, \barint_{B_{2R} }|Du|dx +c \avenorm{ f }^\frac{1}{p-1} _{L(q ,s)(B_R)}.
\end{equation*}
\end{coro}

The above result implies the local version of estimates provided by Boccardo and Gallou\"et~\cite{bgSOLA-cpde} obtained in the case of $s=q$ for problems with $u=0$ on $\partial\Omega$ and parameters as in~\eqref{BG-class} via
\[\| Du \|_{L^{\frac{nq(p-1)}{n-q}}(B_{R/2})} \leq c\, \int_{B_R }|Du|\,dx +c \| f \|^\frac{1}{p-1}_{L^q(B_R)}.\]
Moreover, note that we provide also the estimate in the borderline case $q=np/(np-n+p)$ not included therein, nor by~\cite{BocMarcSola,KiLi} (but covered by Mingione in~\cite{min-grad-est}).

\medskip

The following result extends to the Orlicz setting the estimates provided in~\cite{bgSOLA-cpde} in order to cover the borderline integrability. 
\begin{coro}[Estimates for $L\log L$-data equation]\label{coro:LLogL-est}
 Let $u\in W^{1,1}_{loc}(\Omega)$ be a local SOLA to the equation~\eqref{eq:main} with $G,g$ satisfying (A).  Assume further that $\mu\in L\log L_{loc}(\Omega)$, then
 \[g^\frac{n}{n-1}(|Du|)\in L^1_{loc}(\Omega).\]
 
 Moreover, there exists $c=c(n,G,q,\gamma)$, such that for every ball $B_R\subset\subset\Omega$ we have
 \begin{equation}
\label{eq:LLogL-est} \avenorm{ g( |Du|)}_{L^{\frac{n}{n-1}}(B_{R/2})}\leq c\,g\left( \barint_{B_R }|Du|dx\right) +c \avenorm{\mu }_{L\log L(B_R)}.
\end{equation}
\end{coro}\medskip

Note that Corollary~\ref{coro:LLogL-est} applied to the Zygmund-type function $t^p\log^\beta(1+t)$ leads to the same regularity as the one obtained by Cianchi and Maz'ya~\cite{CiMa} for some other kind of solutions, namely approximable solutions, to the problem corresponding to our~\eqref{eq:main}. Indeed, we have the following remark.

\begin{rem} \label{rem:CiMa}  The result of~\cite{CiMa} is the existence and the Orlicz-Marcinkiewicz regularity of unique approximable solution $u$ to $-\dv A(x,D u)=f$ where
\[A(x,\xi)\xi\geq G(|\xi|) \qquad\text{and}\qquad |A(x,\xi)|\leq c\big(g(|\xi|)+h(x)\big),\]
with $G(|\xi|)$ comparable to $|\xi|^p\log^\beta(|\xi|)$ near infinity such that $1<i_G\leq s_G<\infty$, $G'(t)=g(t)$, and $\wt{G}(h)\in L^1(\Omega)$. Namely, the method leads to
\[D u \in L^{\frac{n(p-1)}{n-1}}(\log L)^{{  \frac{\beta n}{n-1}}}\qquad\qquad 1<p<n,\quad \beta>0.\] Let us point out the obvious misprint in~\cite[Example~3.4]{CiMa} in this line.
\end{rem}

Let us now concentrate on the problems with data being density-driven measures. In order to ensure that the statements are clear and intuitive we present the  results  first in the commonly understood Morrey spaces and below its generalisation to the more sophisticated setting of the Lorentz-Morrey spaces, both defined precisely in Section~\ref{ssec:fn-sp}.

\medskip

We have the following extension of the already mentioned linear Adams theorem~\cite{adams-note}, as well as its $p$-Laplace version~\cite[Theorem~1]{min-grad-est}, retrieving the range of~parameters therein, cf.~\eqref{SAM-class}.

\begin{theo}[Estimates in the Morrey spaces]\label{theo:Morrey-est}
 Let $u\in W^{1,1}_{loc}(\Omega)$ be a local SOLA to the equation~\eqref{eq:main} with $G,g$ satisfying (A) and
\begin{equation}
\label{q-Morrey}i_G\leq \theta\leq n\qquad\text{and}\qquad 1<q\leq\frac{\theta i_G }{\theta s_G-\theta+i_G }.
\end{equation} 
If  $\mu\in L^{q,\theta} $ locally in $\Omega$, then
 \[g(|Du|)\in L^{\frac{\theta q}{\theta-q},\theta}\quad\text{locally in }\ \Omega.\]
 
 Moreover, there exists $c=c(n,G,q,\gamma)$, such that for every ball $B_R\subset\subset\Omega$ we have
 \begin{equation*}
  \avenorm{ g( |Du|) }_{L^{\frac{\theta q}{\theta-q},\theta}(B_{R/2})}\leq c\,R^{\frac{\theta-q}{q}-n} g\left( \barint_{B_{2R} }|Du|dx\right) +c \avenorm{ \mu }_{L^{q,\theta}(B_R)}.
\end{equation*} 
\end{theo}
When we take into account \cite[Remark~7]{min-grad-est}, it implies the following Orlicz version of~\cite[Theorem~12]{min-grad-est} (simplifying to it in the $p$-Laplace case). Compare its form for $\theta=n$ with Corollary~\ref{coro:LLogL-est} above.

\begin{coro}[Borderline Morrey case]\label{coro:Bord-Morrey-est}
 Let $u\in W^{1,1}_{loc}(\Omega)$ be a local SOLA to the equation~\eqref{eq:main} with $G,g$ satisfying (A) and parametrs as in~\eqref{q-Morrey}. If $2\leq\theta\leq n$ and $\mu\in L\log L^\theta $ locally in $\Omega$, then
 \[g(|Du|)\in L^{\frac{\theta}{\theta-1},\theta}\quad\text{locally in }\ \Omega.\]
 
 Moreover, there exists $c=c(n,G,q,\gamma)$, such that for every ball $B_R\subset\subset\Omega$ we have
 \[\avenorm{ g( |Du|) }_{L^{\frac{\theta }{\theta-1},\theta}(B_{R/2})}\leq c\,R^{ \theta-1 -n} g\left( \barint_{B_{2R} }|Du|dx\right) +c \avenorm{ \mu }_{L\log L^\theta(B_R)}.
\]
\end{coro}

Now we present the main accomplishment of the paper extending sharp \cite[Theorem~11]{min-grad-est} (simplifying to it in the $p$-Laplace case) and thus also the classical result by Adams and Lewis~\cite{adams-lewis}. Note that we include the estimate include the upperbound of the rage of parameters $\theta$ and $q$, as well as the Marcinkiewicz case ($s=\infty$).

\begin{theo}[Estimates in the Lorentz-Morrey spaces]\label{theo:Lor-Mor}
 Let $u\in W^{1,1}_{loc}(\Omega)$ be a local SOLA to the equation~\eqref{eq:main} with $G,g$ satisfying (A), parametrs $\theta,q$ be as in~\eqref{q-Morrey} and $s\in(0,\infty]$. If  $\mu\in L^{\theta}(q,s) $ locally in $\Omega$, then
 \[g(|Du|)\in L^\theta\left(\frac{n q}{n-q},\frac{ns}{n-q}\right)\quad\text{locally in }\ \Omega.\]
 
 Moreover, there exists $c=c(n,G,q,s)$, such that for every ball $B_R\subset\subset\Omega$ we have
\[\avenorm{g( |Du |)}_{L^\theta\left(\frac{\theta q}{\theta-q},\frac{\theta s}{\theta-q}\right)(B_{R/4})} \leq cg\left(\barint_{B_R}|Du |dx\right)+c\avenorm{\mu 
 }_{L^\theta(q,s)(B_R)}.\]
\end{theo}

\section{Preliminaries}\label{sec:prelim}

\subsection{Notation and basics}
In the following we shall adopt the customary convention of denoting by $c$ a constant that may vary from line to line. Sometimes to skip rewriting a constant, we use $\lesssim$. By $a\simeq b$, we mean $a\lesssim b$ and $b\lesssim a$. By $B_R$ we shall denote a ball usually skipping prescribing its center, when its is not important. Then by $cB=B_{cR}$ we mean then a ball with the same center as $B_R$, but with rescaled radius $cR$. 

 We consider the (restricted) maximal function operators related to a ball 
\begin{equation}\label{MDu-Mmu-def}
\MDu(x)=\sup\limits_{\substack{x\in B_R\\ B_R\subset 2B_0}}\barint_{B_R}|Du(y)|dy,\qquad \Mmu(x)=\sup\limits_{\substack{x\in B_R\\ B_R\subset 2B_0}} {|\mu(B_R)|}{R^{\frac{1}{n}-1}}. 
\end{equation}

\subsection{The Orlicz setting}\label{sec:Or}
We study the solutions to PDEs in the Orlicz-Sobolev spaces equipped with a modular function $B$ - an increasing and convex function satisfying~\begin{equation}
\label{iBsB}  1<i_B=\inf_{t>0}\frac{tB'(t)}{B(t)}\leq \sup_{t>0}\frac{tB'(t)}{B(t)}=  s_B<\infty.
\end{equation}

\begin{defi}[Orlicz-Sobolev space]\label{def:OrSob:sp} 

 By the Orlicz space ${L}_B(\Omega)$  we understand the space of measurable functions endowed with the Luxemburg norm 
\[||f||_{L_B}=\inf\left\{\lambda>0:\ \ \int_\Omega B\left( \frac{|f(x)|}{\lambda}\right)\,dx\leq 1\right\}.\]
 We define the Orlicz-Sobolev space  $W^{1,B}(\Omega)$  as follows
\begin{equation*} 
W^{1,B}(\Omega)=\big\{f\in L_B(\Omega):\ \ D f\in L_B(\Omega)\big\},
\end{equation*}endowed with the norm
\[
\|f\|_{W^{1,B}(\Omega)}=\inf\bigg\{\lambda>0 :\ \   \int_\Omega B\bigg( \frac{|f|}{\lambda}\bigg)dx+\int_\Omega B\bigg( \frac{|D f|}{\lambda}\bigg)dx\leq 1\bigg\} 
\]
and  by $W_0^{1,B}(\Omega)$ we denote a closure of $C_c^\infty(\Omega)$ under the above norm. 
\end{defi}
Directly from the definition of the norm we get the following information.
\begin{lem}\label{lem:norm} If ${B}$ is a Young function $B$ and $f\in L_B(\Omega)$, then  for every $\epsilon\in(0,1)$ we have
\[\epsilon \|f\|_{L_B(\Omega)}\leq \int_{\Omega}{B}(\epsilon |f|)\,dx+1.\]
\end{lem}

In the functional analysis of the Orlicz setting an important role is played by  $\wt{B}$ -- the defined below  complementary~function (called also the Young conjugate, or the Legendre transform) to a function  $B:\r\to\r$. The complementary~function is given by the following formula
\[\wt{B}(t)=\sup_{s>0}(s\cdot t-B(s)).\]

\begin{lem}[\cite{rao-ren}]\label{lem:conjugate} If $\wt{B}$ is a complementary function to a Young function $B$, then  $\wt{B}$ is also a Young function. Moreover, we have
\[t\leq B^{-1}(t)\wt{B}^{-1}(t)\leq 2t.\]
\end{lem} 

\subsubsection*{Growth}
 We would like to comment the growth condition under which we work. The typical assumption on the growth within the Orlicz classes is the following one.

\begin{defi}[$\Delta_2$-condition, doubling condition]\label{def:D2}
 We say that a function $B:\r\to\r$ satisfies $\Delta_2$-condition if there exists a constant $c_{\Delta_2}>0$ such that $B(2s)\leq c_{\Delta_2}B(s).$
\end{defi} 

It describes the speed and regularity of the growth. For instance when $B(s) = (1+|s|)\log(1+|s|)-|s|$, its complementary function is  $\widetilde{B}(s)= \exp(|s|)-|s|-1$. Then $B\in\Delta_2$ and   $\widetilde{B}\not\in\Delta_2$.

 We point out that the condition~\eqref{iBsB}  is equivalent to $B,\wt{B}\in\Delta_2$,~\cite[Section~2.3, Theorem~3]{rao-ren}. Indeed, if $s_B<\infty$ then $B\in\Delta_2$, whereas $i_B>1$ entails the $\Delta_2$-condition imposed on $\wt{B}$. It also implies a to comparison with power-type functions in the sense that when $B$ satisfies~\eqref{iBsB}, then
\begin{equation}
\label{B-power-compar}
\frac{B(t)}{t^{i_B}}\quad\text{is non-decreasing}\qquad\text{and}\qquad\frac{B(t)}{t^{s_B}}\quad\text{is non-increasing}.
\end{equation}
 On the other hand, \cite[Example~3.2]{CGZG} shows that being trapped between two power-type functions is not enough for a convex and continuous function to satisfy the $\Delta_2$-condition.  Therefore, the assumption~\eqref{iBsB} is more restrictive than assumption on $p,q$--growth, as it influence on both regularity of the growth and its speed.

\begin{rem}[\cite{adams-fournier}] Since condition~\eqref{iBsB} imposed on $B$ implies $B,\wt{B}\in\Delta_2$, the Orlicz-Sobolev space $W^{1,B}(\Omega)$ we deal with is separable and reflexive.
\end{rem}
 
\begin{lem}[\cite{rao-ren}]\label{lem:D2} If $B\in\Delta_2$, then $B(r+s)\lesssim B(r)+B(s)$.
\end{lem} 
 
\subsubsection*{Embeddings}
For Sobolev--Orlicz spaces expected embedding theorems hold true. We distinguish two possible behaviours of $B$
\begin{equation}
\label{intB}
\int^\infty\left(\frac{t}{B(t)}\right)^\frac{1}{n-1}dt=\infty \qquad\text{and}\qquad
\int^\infty\left(\frac{t}{B(t)}\right)^\frac{1}{n-1}dt<\infty,
\end{equation}
which roughly speaking describe slow an fast growth of $B$ in infinity, respectively. The condition imposing slow growth of $B$, namely \eqref{intB}$_1 $, corresponds to  the case of $p$-growth with $p\leq n$. Then we expect $W_0^{1,B}\hookrightarrow{} L_{\hat{B}}$ with $\hat{B}$ growing faster than $B$ (it is presented below). In the case of quickly growing modular function, i.e. when~\eqref{intB}$_2$ holds (corresponding to $p>n$), it holds that $W_0^{1,B}\hookrightarrow{}L^\infty$. Below we give details.  

We apply the optimal embeddings due to~\cite{Ci96-emb}, where  the Sobolev inequality is proven under the restriction 
\begin{equation}\label{int0B}\int_0\left(\frac{t}{B(t)}\right)^\frac{1}{n-1}dt<\infty, 
\end{equation} 
concerning the growth of $B$ in the origin. Nonetheless, the properties of $L_B$ depend on the behaviour of $B(s)$ for large values of $s$ and~\eqref{int0B} can be easily by-passed in application. When we consider 
\begin{equation}\label{BN}
H_n(s)=\left(\int_0^s\left(\frac{t}{B(t)}\right)^\frac{1}{n-1}dt\right)^\frac{n-1}{n} \qquad \text{and}\qquad
B_n(t)=B(H_n^{-1}(t)), 
\end{equation}
the following result follows.

\begin{theo}[Sobolev embedding, \cite{Ci96-emb}]\label{theo:Sob-emb} Let $\Omega\subset\rn$, $n>1$, be a bounded open set.
\begin{itemize}
\item[(slow)] If $B$ is a Young function satisfying \eqref{int0B} and \eqref{intB}$_1 $, then there exists a constant $c_s=c_s(n)$, such that for every $u\in W_0^{1,B}(\Omega)$ it holds that \[\int_\Omega B_n\left(\frac{|u|}{c_s\big(\int_\Omega B(|Du|)dx\big)^\frac{1}{n}}\right)dx\leq \int_\Omega B(|Du|)dx.\]

\item[(fast)]  If $B$ is a Young function satisfying \eqref{intB}$_2 $, then then there exists a constant $c(n)$, such that for every $u\in W_0^{1,B}(\Omega)$ it holds that \[\|u\|_{L^\infty(\Omega)}\leq \|Du\|_{L_B(\Omega)}.\]
\end{itemize}
\end{theo}

\subsection{Notion of SOLA}\label{ssec:SOLA}
 Investigating the general elliptic Dirichlet problem 
\begin{equation} \label{intro:ell:f}
- \dv \,a(x,D u)= f
\end{equation} 
involving $a$~from an Orlicz class and on the right--hand side data merely integrable or in the space of~measures special notion of solutions need to be introduced. Indeed, the weak formulation of~\eqref{intro:ell:f}, i.e.
\[\int_\Omega a(x,D u)D\vp\,dx=\int_\Omega f\vp\,dx,\]
is expected to hold for every $\vp$ in the Orlicz-Sobolev space $W_0^{1,G}(\Omega)$. There are at least three different classical approaches to this kind of problems keeping uniqueness even under weak assumptions on the data. The notion of renormalized solutions  appeared first in~\cite{diperna-lions}, whereas the entropy solutions comes from~\cite{bbggpv,dall}. The SOLA are studied starting from~\cite{bgSOLA-jfa,bgSOLA-cpde,bgo}. See also~\cite{ACM,FS,Rak} for other classical results.

To consider the datum $f$ not belonging to the dual space, we adopt the notion of SOLA. However, under certain restrictions the mentioned notions coincide~\cite{KiKuTu}, which suggests that the gradient estimates we obtain for SOLA, can be shared by the other types of solutions.

\begin{defi}[Local SOLA]\label{def:SOLA}
A function $u\in W^{1,1}_{loc}
(\Omega)$ is called a local {SOLA} to~\eqref{intro:eq:main} if  problems \begin{equation}
\label{eq:main-k-for-SOLA}
- \dv\, a(x, Du_k ) = f_k=\mu_k\in L^\infty(\Omega)
\end{equation} with $\mu_k\to\mu\in{\cal M}(\Omega)$ $*$--(locally)--weakly in the sense of measures, that is $
\lim_{k\to\infty}\int_\Omega \vp\,f_k\,dx=\int_\Omega\vp\,d\mu$ for every  continuous function $\vp$ with compact support in $\Omega$, and satisfying
$\limsup_{k\to\infty} |\mu_k |(B) \leq |\mu|(B)$
for every ball $B\subset\Omega$  
have solutions  $\{u_k\}_k\subset W^{1,G}_{loc}
(\Omega)$ such that \[u_k\xrightarrow[k\to\infty]{} u\quad\text{ strongly\ \ in\ \  }W^{1,1}_{loc}
(\Omega)\qquad\text{and}\qquad g(|D u_k|)\xrightarrow[k\to\infty]{}g(|D u|)\quad\text{strongly\ in \  }L^{1}_{loc}(\Omega).\]
\end{defi}
For the existence of such solutions results see e.g.~\cite{CiMa} and considerations in~\cite[Section 7]{Baroni-Riesz}. Note that the uniqueness is kept within this notion of solutions if only the data $\mu$ is locally integrable, while for general measure data it is an open problem.

\section{Auxiliary results}\label{sec:aux}
In order to compare the properties of solutions to our main equation to the solutions to the homogeneous equation (i.e. null-data one) first we prove some integrability results for solutions to homogeneous problem itself and then we infer comparison estimates.

\subsection{Homogeneous problem}\label{ssec:homog}

This subsection is devoted to various estimates for $v$ solving the homogeneous problem
\begin{equation}\label{eq:homog} 
-{\dv}\, a(x,Dv)=0.
\end{equation} 

\begin{prop}[Estimates for the homogeneous problem]\label{prop:homo-problem} Suppose $B_{2R}\subset\subset A\subset\rn$, $A$ is a bounded set, and $v \in W^{1,G} (A)$ is a weak solution to~\eqref{eq:homog} on $A$, where $a:\rn\to\rn$ and $G,g:[0,\infty)\to[0,\infty)$ satisfy Assumption~(A). Then
\begin{itemize}
\item[(i)]  there exists a constant $c=c(n,\nu,L,s_G)$, such that\begin{equation}
\label{rev-Hold}
\barint_{B_{R }}G(|Dv|)\, d x \leq c\,G\left(\barint_{B_{2R} }   
|Dv|\, d x \right),
\end{equation}

\item[(ii)] then there exist $ {c}_1,{c}_2>0$ and $\chi>1$, such that
\begin{equation}\label{higher-int}\barint_{B_R} G^{\chi}(|D v|)\,dx \leq  {c}_1\, G^{\chi}\left(\barint_{B_{2R}}   |D v|  \,dx\right)+ {c}_2,\end{equation}

\item[(iii)]there exists $c>0$, such that  
\begin{equation}
\label{inq:cacc} 
\int_{B_{ R}}G(|Dv|)\,dx\leq c \int_{B_{2R}}G\left(\frac{|v-(v)_{B_R}|}{R}\right)
dx,
\end{equation}

\item[(iv)] for $\vr<R $,  there exist $c,\beta>0$, such that 
\begin{equation}
\label{inq:Morrey-continuity} \barint_{B_{\vr }}g(|Dv|)dx\leq c\left(\frac{\vr}{R}\right)^{-\beta} \barint_{B_{2R}}g (|Dv|)dx.
\end{equation}
\end{itemize}

\end{prop}
\begin{proof} $ $

\textit{(i)} The reverse  H\"older  inequality~\eqref{rev-Hold} is obtained as in~\cite[Lemma 4.2]{Baroni-Riesz} for the problem with the leading part of the operator independent of $x$, but the proof is essentially the same. We provide its short version for the sake of completeness. For more comments see~\cite{Baroni-Riesz}. According to~\cite[(1.11)]{Cianchi-Fusco} we have
\begin{equation}
\label{in:ci-fu}
\barint_{B_{\vr/2(y) }}G(|Dv|)\, d x \leq c\,(G\circ S^{-1})\left(\barint_{B_{\vr(y)} }S(   
|Dv|)\, d x \right)
\end{equation}
with $c=c(n,\nu,L,s_G)$, $B_\vr(y)\subset B_R$, and function $S(t):=G^{\frac{n-1}{n}}(t)t^{1/n}$.

It suffices to prove the inequality for $R=1$. Indeed, the general case can be deduced then by considering $\wt{v}(x)=v(x_0+Rx)/R$ solving $-\dv\,a(x,D\wt{v})=0$ on $B_1(0)$. Then having~\eqref{rev-Hold} for $\wt{v}$ and rescaling back, we get the final claim.

Thus, from now on $R=1$. Let us fix $r\leq 1$, $\alpha\in (0,1)$, $y\in B_{\alpha r}(x_0)$, and $\vr=(1-\alpha) r$. Note that $B_{\vr}(y)\subset B_{r}(y).$ Moreover, using H\"older's inequality we have
\[
\barint_{B_{\vr}}S(|Dv|)\, d x =  \barint_{B_{\vr(y)} } G^\frac{n-1}{n}(   
|Dv|)\cdot|Dv|^\frac{1}{n}\, d x \leq \left(\barint_{B_{\vr(y)} } G (   
|Dv|) \, d x \right)^\frac{n-1}{n}\left(\barint_{B_{\vr(y)} } |Dv|\, d x \right)^\frac{1}{n}. \]
We aim at estimating the right-hand side above using the Young inequality. Let us consider a Young function $C(t):=S(t^n)$ and its complementary function $\wt{C}$. Then $i_C\leq 2n-1$ and $s_C\leq (n-1)s_G+1$.  If necessary for ensuring convexity, we shall consider $C_1(t)=\int_0^t C(s)/s\,ds\simeq C(t)$.  We obtain
\[
\barint_{B_{\vr}}S(|Dv|)\, d x \leq \epsilon\wt{C}\left(\left[\barint_{B_{\vr(y)} } G (   
|Dv|) \, d x \right]^\frac{n-1}{n}\right)+c(\epsilon)C\left(\left[\barint_{B_{\vr(y)} } |Dv|\, d x \right]^\frac{1}{n}\right). \]

Let us concentrate on the first term on the right-hand side above. When we denote $T(t):=S^n(t)$, we have
\[\wt{C}\big(\alpha^\frac{n-1}{n}\big)=
\sup_{s>0} \left(\alpha^\frac{n-1}{n}s-S(s^n)\right)
\lesssim 
\left[ \sup_{\sigma>0} \left(\alpha^{n-1}\sigma-S^n(\sigma)\right)\right]^\frac{1}{n}=
\big[\wt{T} \left(\alpha^{n-1} \right)\big]^\frac{1}{n},\]
where we can assume that $\alpha^\frac{n-1}{n}s\geq S(s^n)$. Lemma~\ref{lem:G*} implies then that
\[\wt{T}\left(G^{n-1}(\tau)\right)=\wt{T}\left(T(\tau)/\tau\right)\leq  {T}\left( \tau \right)=S^n(\tau).\]
Applying it to $\tau =G^{-1}(\alpha)$ within our set of $s$ we have
\[\wt{T}\left(\alpha^{n-1}\right)\leq \left(S \circ G^{-1}(\alpha)\right)^n,\]
where, finally, putting $\alpha=\int_{B_r} G(|Dv|)\,dx$, we get
\[  \wt{C}\left(\left[\barint_{B_{\vr(y)} } G (   
|Dv|) \, d x \right]^\frac{n-1}{n}\right)\lesssim (S\circ G^{-1})\left( \barint_{B_{\vr(y)} } G(|Dv|)\, d x  \right). \]

Summing up the above observations (recall $\vr=(1-\alpha)r$) and choosing appropriate $\epsilon>0$ we obtain
\[
\barint_{B_{(1-\alpha)r}}G(|Dv|)\, d x \leq \frac{1}{2} \barint_{B_{(1-\alpha)r} } G (   
|Dv|) \, d x+c G\left( \barint_{B_{(1-\alpha)r}} |Dv|\, d x\right), \]
and consequently
\[
\int_{B_{(1-\alpha)r}}G(|Dv|)\, d x \leq \frac{1}{2} \int_{B_{(1-\alpha)r} } G (  
|Dv|) \, d x+\frac{c}{[(1-\alpha)r]^{ns_G}} G\left( \int_{B_{(1-\alpha)r}} |Dv|\, d x\right).\]
Since the ball $B_{\alpha r}$ can be covered by family of balls included in $B_r$ such that only a finite and independent of $\alpha$ number of balls of double radius intersect, we set $\alpha r= s<r$ and infer that
\[
\int_{B_{s}}G(|Dv|)\, d x \leq \frac{1}{2} \int_{B_{ r} } G (  
|Dv|) \, d x+\frac{c}{ (r-s)^{ns_G}} G\left( \int_{B_{r}} |Dv|\, d x\right).\]
Now Lemma~\ref{lem:absorb1} gives the claim for $R=1$ and, as noticed above, the proof is complete.

\bigskip

\textit{(ii)} Higher inegrability~\eqref{higher-int} can be obtained as a direct consequence of reverse H\"older inequality~\eqref{rev-Hold} and~\cite[Proposition~2.1]{Fusco-Sbordone}.

\bigskip

\textit{(iii)} To get Caccioppoli estimate~\eqref{inq:cacc} let us take a hat-function $\eta\in C_c^\infty(B_{2R})$, such that $\mathds{1}_{B_R}\leq\eta\leq \mathds{1}_{B_{2R}}$ and $|D\eta|\leq c/R$. We test~\eqref{eq:comp-map} with $\xi=\eta^q(v-(v)_{B_R})$, where $q>1$ is to be chosen soon, to get
\[\langle a(x,Dv),\eta^q Dv\rangle=-\langle a(x,Dv),q\eta^{q-1}(v-(v)_{B_R})D\eta \rangle.\]
Therefore, due to monotonicity of $a$ and the Cauchy-Schwartz inequality we have
\begin{equation}
\label{cac-inq-test}
\int_{B_{2R}}G(|Dv|)\eta^qdx\leq c\int_{B_{2R}}g(|Dv|)\eta^{q-1}\frac{|v-(v)_{B_R}|}{R}
dx.
\end{equation}
We choose $q$ big enough to satisfy $(1+i_G)\geq q'$ and notice that then we have $G(\eta^{q-1}t)\leq c \eta^qG(t)$. Therefore due to Lemma~\ref{lem:G*}  we get\[\wt{G}(\eta^{q-1}g(t))\leq c \eta^q \wt{G}(g(t))\leq c \eta^q G(t).\] 
Then, the Young inequality with an $N$-function $G$ and its conjugate $\wt{G}$ applied on the right-hand side of~\eqref{cac-inq-test} enables us to write
\begin{equation*}
\begin{split}\int_{B_{2R}}g(|Dv|)\eta^{q-1}\frac{|v-(v)_{B_R}|}{R}
dx&\leq \ve \int_{B_{2R}}\wt{G}(\eta^{q-1}(|Dv|)) dx+ c_\ve \int_{B_{2R}}G\left(\frac{|v-(v)_{B_R}|}{R}\right)
dx\\
&\leq \ve c\int_{B_{2R}} \eta^{q }G(|Dv| ) dx+ c_\ve \int_{B_{2R}}G\left(\frac{|v-(v)_{B_R}|}{R}\right)
dx.
\end{split}
\end{equation*}
with arbitrary $\ve<1$. Combining it with~\eqref{cac-inq-test}, choosing $\ve$ small enough to absorb the term, and noticing that $\eta\geq \mathds{1}_{B_R},$ we obtain~\eqref{inq:cacc}.
 
 \bigskip

\textit{(iv)} To get~\eqref{inq:Morrey-continuity}, we start with the proof that there exist $\beta_1>1,\ c>0$, such that for every 
 $B_{R}\subset A$ we have
\begin{equation}
\label{continuity-Morrey-G}\barint_{B_{\vr }}G(|Dv|)dx\leq c\left(\frac{\vr}{R}\right)^{- \beta_1 } \barint_{B_{R}}G (|Dv|)dx .
\end{equation} 
Therefore,~\eqref{inq:cacc} together with the Sobolev embedding given in Proposition~\ref{lem:Sob} (note that $G^{1/n'}$ is convex due to $i_G>n'$) give
\[\barint_{B_{\vr }}G(|Dv|)dx\leq c\, \barint_{B_{2\vr}}G\left(\frac{|v-(v)_{B_{\vr}}|}{\vr}\right)
dx\leq c\left(\barint_{B_{2\vr}}G^\frac{1}{n'}(|Dv|)dx\right)^{n'}. \]
We estimate further the right-hand side above extending the domain of integration and using H\"older's inequality
\[\left(\barint_{B_{2\vr}}G^\frac{1}{n'}(|Dv|)dx\right)^{n'}\leq c\left(\frac{\vr}{R}\right)^{-nn'}\left(\barint_{B_{2R}}G^\frac{1}{n'}(|Dv|)dx\right)^{n'}\leq c\left(\frac{\vr}{R}\right)^{-\beta_1} \barint_{B_{2R}}G (|Dv|)dx.\]  Summing up the above inequalities we arrive at~\eqref{continuity-Morrey-G}.

To get the final claim we apply Jensen's inequality, extend the domain of integration and apply further~\eqref{continuity-Morrey-G}  we have
\[\barint_{B_{\vr }}g(|Dv|)dx\leq c g \circ G^{-1}\left(\barint_{B_{\vr }}G(|Dv|)dx\right)\leq c g \circ G^{-1}\left(c\left(\frac{\vr}{R}\right)^{-\beta_1}   \barint_{B_{2R }} G(|Dv|) dx \right).\]
Then~\eqref{rev-Hold} enables to arrive at
\[ \barint_{B_{\vr }}g(|Dv|)dx \leq  c \left(\frac{\vr}{R}\right)^{-\frac{\beta_1(i_G-1)}{s_G}} g\left(c\ \barint_{B_{2R}} |Dv| dx\right)  \leq c \left(\frac{\vr}{R}\right)^{-\frac{\beta_1(i_G-1)}{s_G}}\barint_{B_{2R}} g(|Dv|) dx,\]
where we also used that $t\mapsto g\circ G^{-1}(t)$ is concave and grows faster than $t^{-(i_G-1)/s_G}$, then the Jensen inequality, and write $\beta= \beta_1 (i_G-1)/s_G>0$.
\end{proof}

\subsection{Comparison estimates and direct consequences}
We provide a comparison estimate between solution to~\eqref{eq:main} and $v\in u+W^{1,G}_0(B_R)$ solving
\begin{equation}\label{eq:comp-map}
\left\{\begin{array}{ll}
-{\dv}\, a(x,Dv)=0&\text{ in }B_{R},\\
v=u&\text{ on }\partial B_{R}.
\end{array}\right.
\end{equation} For existence and uniqueness for this problem we refer to~\cite[Lemma 5.2]{Lieb91}.

To get the comparison estimate we modify the proof of~\cite[Lemma~5.3]{Baroni-Riesz} to capture $x$-dependence of vector field $a$. We note that here we meet essential obstacle to go include the case of $p<2$. It is solved in the case of $p$-Laplacian in~\cite{ku-min-univ} (see also~\cite{KuMi}), though the Orlicz analogue is substantially more technical and is not yet proven.

\begin{prop} \label{prop:comp-B-xi}
Suppose $a:\Omega\times\rn\to\rn$ satisfies Assumption (A). If $u\in W^{1,G}(\Omega)$ is a local SOLA to~\eqref{eq:main} and $v\in u+W^{1,G}_0(B_R)$ is a weak solution to~\eqref{eq:comp-map} on $B_R$, then   there exist a constant $c=c(n,\nu,s_G)>0$, such that
\begin{equation}\label{eq:comp-est}
 \barint_{B_R} g  (|Du-D v|)\,dx \leq c\left(\frac{|\mu|(B_R)}{R^{n-1}}\right) .
\end{equation}
\end{prop}

\begin{proof} In order to get~\eqref{eq:comp-est}, we rescale the equation to the unit ball, provide estimate in the cases of slowly and fastly growing $G$ separately, and then rescaling back we arrive at the claim.

\medskip

\noindent\textsc{Step 1. Rescaling. }

 Let us note that if $\mu(B_R)=0,$ then monotonicity of the vector field $a$ implies that $u=v$ and~\eqref{eq:comp-est} trivially follows. Otherwise, i.e. when $\mu(B_R)\neq 0,$ we fix $A=g^{-1}(|\mu|(B_R)R^{1-n})$ and rescale 
 \begin{equation}
\label{resc-comp} 
 \begin{array}{llll}
\bu(x)=&\frac{u(x_0+Rx)}{AR},&\qquad 
\bv(x)=&\frac{v(x_0+Rx)}{AR},\\\\
\ba(x)=&\frac{a(x,Az)}{|\mu|(B_R)}R^{n-1},&\qquad 
\bv(x)=&\frac{\mu(x_0+Rx)}{|\mu|(B_R)}R^n.
 \end{array} 
 \end{equation}
Then $|\bmu|(B_1)=1$ and the growth of $\ba$ is governed by $\bar{g}(t)=g(At)R^{n-1}/|\mu|(B_R)$. Indeed, due to Assumption (A), we have
\[
\langle\partial_z \ba(x,z)\lambda,\lambda\rangle=\frac{A R^{n-1}}{|\mu|(B_R)}\langle\partial_{Az}  a(x,Az)\lambda,\lambda\rangle\geq \nu\frac{AR^{n-1}}{|\mu|(B_R)}\frac{g(A|z|)}{A|z|}|\lambda|^2.\] 

Then~\eqref{eq:main} and~\eqref{eq:comp-map} implies
\begin{equation}\label{eq:diff-comp-est}
 -\dv\big( \ba(x,D\bu)-\ba(x,D \bv)\big)=\bmu\qquad\text{in }\ B_1 
\end{equation}
admitting the weak formulation

\begin{equation}\label{eq:weak-diff-comp-est}
 \int_{B_1}\langle \ba(x,D\bu)-\ba(x,D \bv),D\vp\rangle\,dx=\int_{B_1}\vp\, d\bmu.
\end{equation}

Our aim now is to provide estimate
\begin{equation}\label{eq:resc-comp-est}
 \int_{B_1} \bar{g}  (|D\bu-D \bv|)\,dx \leq c(n,\nu,s_G).
\end{equation}

\medskip

\noindent\textsc{Step 2. Measure data estimates. } As typically in the Orlicz setting, we distinguish two types of growth of $G$ for which the Sobolev embedding (Theorem~\ref{theo:Sob-emb}) has different form. See comments in Section~\ref{sec:Or}.

\smallskip

\textsc{Step 2.1. Slowly growing $G$. } 

\noindent We consider \[\gb(t)=\int_0^t\frac{\bg(s)}{s}ds.\]
Notice that $(i_G-1)\gb(t)\leq \bg(t)\leq (s_G-1)\gb(t)$ for $t>1$. In order to use Sobolev's embedding we define 
\begin{equation}
\label{g0vt}{g}^0(t)=t\gb(1)\mathds{1}_{[0,1]}(t)+\gb(t)\mathds{1}_{(1,\infty)}(t)\quad\text{and}\quad  \vt=c_s\left(\int_{B_1} g^0(|D\bu-D \bv|)\,dx\right)^\frac{1}{n}.\end{equation}
We note that withous loss of generality $\vt\geq 1$. What is more, Young's inequality implies that
\begin{equation}
\label{vt-est}\vt\leq  {\epsilon_1}  \int_{B_1} \bg (|D\bu-D \bv|)\,dx+c( {\epsilon}_1 )\end{equation}
with arbitrary $\epsilon_1>0$ to be chosen later.

For any $k>0$ and $\sigma\in\r$ let us denote
\[T_k(\sigma)=\max\{-k,\min\{k,\sigma\}\}\quad\text{and}\quad \Phi_k(\sigma)=T_1(\sigma-T_k(\sigma)).\]
At first we use
\[\vp=T_k\left(\frac{\bu-\bv}{\vt}\right)\in W^{1,\bG}_0(B_1)\cap L^\infty(B_1)\]
as a test function in~\eqref{eq:weak-diff-comp-est}. Note that
\[D\vp=\frac{D(u-v)}{\vt}\mathds{1}_{C_k},\qquad\text{where}\qquad C_k=\left\{x\in B_1:\ \frac{|u(x)-v(x)|}{\vt}\leq k\right\}.\]
Notice that since $\bg$ satisfies~\eqref{ig-sg}, due to~\cite{DieEtt}, we have
\[\bG(|\xi_1-\xi_2|)\lesssim \frac{\bg(|\xi_1-\xi_2|)}{|\xi_1-\xi_2|}|\xi_1-\xi_2|^2\lesssim \frac{\bg(|\xi_1|+|\xi_2|)}{|\xi_1|+|\xi_2|}|\xi_1-\xi_2|^2.\]
Consequently,
\[\frac{c}{\vt}\int_{C_k}\bG(|D\bu-D\bv|)dx\leq \frac{1}{\vt}\int_{C_k}\langle \ba(x,D\bu)-\ba(x,D \bv),D\bu-D\bv\rangle\,dx=\int_{B_1}\langle \ba(x,D\bu)-\ba(x,D \bv),D\vp\rangle\,dx.\]

On the other hand, when we recall that $|\mu |(B_1)=1$, we observe
\[\left|\int_{B_1}T_k\left(\frac{u-v}{\vt}\right)d\mu\right|
\leq \int_{B_1} k\, d|\mu | =k |\mu |(B_1)=k.\]
 
Therefore,
\begin{equation}
\label{G-Ck} \int_{C_k}\bG(|D\bu-D\bv|)dx\leq \frac{\vt}{c}\int_{B_1}\langle \ba(x,D\bu)-\ba(x,D \bv),D\vp\rangle\,dx=\frac{\vt}{c}\int_{B_1}\vp\, d\bmu\leq Ck\vt.\end{equation}

Using the same arguments with $\vp=\Phi_k((u-v)/\vt)\in W^{1,G}_0(B_1)\cap L^\infty(B_1)$ as a test function in~\eqref{eq:weak-diff-comp-est} we get
\begin{equation}
\label{G-Ck+1-Ck} \int_{C_{k+1}\setminus C_k}\bG(|D\bu-D\bv|)dx\leq \frac{\vt}{c}\int_{B_1}\langle \ba(x,D\bu)-\ba(x,D \bv),D\vp\rangle\,dx=\frac{\vt}{c}\int_{B_1}\vp\, d\bmu\leq C\vt.\end{equation}

Let us note that $t\mapsto \gb\circ \bG^{-1}(t)$ is increasing and concave. Indeed, direct computations show that\[\begin{split}\frac{d}{dt}\gb(\bG^{-1}(t))=\frac{1}{\bG^{-1}(t)}>0\qquad\text{and}\qquad
\frac{d^2}{dt^2}\gb(\bG^{-1}(t)) =-\frac{1}{\bg(\bG^{-1}(t))(\bG^{-1}(t))^2}<0.\end{split}\] 
Let us note that for the function $H(s):=1/\bG(1/s)$ it holds that \[t\mapsto H^{-1}(t)=\frac{1}{\bG^{-1}(1/t)}\] is increasing, concave, and satisfies $\Delta_2$-condition. Therefore, due to Lemma~\ref{lem:D2} for every $t,s>0$
\[\begin{split}  H^{-1}(t+s)\lesssim H^{-1}(t)+H^{-1}(s).
\end{split}\]
and
for $s>1$
\begin{equation}
\label{tildH-l-H}\wt{H}^{-1}(s)\leq 2 H(s).\end{equation}

Since $g$ satisfies~\eqref{ig-sg}, we notice that
\begin{equation}
\label{gbG-tildH}\gb\circ \bG^{-1}\left(\frac{1}{t}\right)\simeq \frac{1}{t}H^{-1}(t).\end{equation}

Concavity of $\gb\circ \bG^{-1}$ enables to apply Jensen's inequality in~\eqref{G-Ck} and~\eqref{G-Ck+1-Ck} and then use the above observation leading to
\begin{equation}
\label{gbH}\begin{split} \int_{C_k}\gb(|D\bu-D\bv|)dx&\leq |C_k|\gb\circ \bG^{-1}\left(\barint_{C_k}\bG(|D\bu-D\bv|)dx\right)\\
&\lesssim \gb\circ \bG^{-1}\left(\frac{k\vt}{|C_k|}\right)\lesssim { k\vt}H^{-1}\left(\frac{|C_k|}{k\vt}\right)\end{split}\end{equation}
and
\[\begin{split}
\int_{C_{k+1}\setminus C_k}\gb(|D\bu-D\bv|)dx&\leq |C_{k+1}\setminus C_k|\cdot\gb\circ \bG^{-1}\left(\barint_{C_{k+1}\setminus C_k}\bG(|D\bu-D\bv|)dx\right)\\
&\lesssim \gb\circ \bG^{-1}\left(\frac{\vt}{|C_{k+1}\setminus C_k|}\right)\lesssim {\vt}H^{-1}\left(\frac{|C_{k+1}\setminus C_k|}{\vt}\right).\end{split}\]
Altogether we have
\begin{equation}
\label{est-I}\begin{split} \int_{B_1}\gb(|D\bu-D\bv|)dx&\leq \int_{C_1}\gb(|D\bu-D\bv|)dx+\sum_{k=1}^{\infty}\int_{C_{k+1}\setminus C_k}\gb(|D\bu-D\bv|)dx\\
&\lesssim {\vt} H^{-1}\left(\frac{|B_1|}{\vt}\right)+{\vt}\sum_{k=1}^{\infty}H^{-1}\left(\frac{|C_{k+1}\setminus C_k|}{\vt}\right) =:I_1+I_2.\end{split}\end{equation}

Let us concentrate on $I_1$. Due to~\eqref{gbG-tildH}, $\bg(t)\simeq \bG(t)/t$, and since we suppose $\vt\geq 1$, we get
\[I_1={\vt} H^{-1}\left(\frac{|B_1|}{\vt}\right)\lesssim 
\vt\leq  {\epsilon}_1  \int_{B_1} \bg (|D\bu-D \bv|)\,dx+c( {\epsilon}_1 ).\]

In order to estimate $I_2$ we shall apply Sobolev's embedding (Theorem~\ref{theo:Sob-emb}) with $B=\gb$. For this we define \[\gbn(t)=g^0\circ H_n^{-1}(t),\]
where $H_n$ is given by \eqref{BN} with $B=\gb$. Notice that when $t>1$, then
\begin{equation}
\label{H-est}\begin{split}H_n^{-1}(t)&\geq \left[\int_0^1 \left(\frac{s}{\gb(s)}\right)^\frac{1}{n-1}ds\right]^\frac{n-1}{n}=\left[\frac{1}{\gb(1)}\right]^\frac{1}{n}\gtrsim 1,\\
\gbn(k)&=g^0\circ H_n^{-1}(k)=\gb\circ H_n^{-1}(k).\end{split}\end{equation}
The convexity of $\gbn$ implies that
\[|C_{k+1}\setminus C_k|\leq \frac{1}{\gbn(k)}\int_{C_{k+1}\setminus C_k} \gbn\left(\frac{|u-v|}{\vt}\right)dx.\]
Therefore we can estimate
\[\begin{split}I_2=\sum_{k=1}^{\infty}H^{-1}\left(\frac{|C_{k+1}\setminus C_k|}{\vt}\right)&\leq \sum_{k=1}^{\infty}H^{-1}\left(\frac{1}{\gbn(k)\vt}\int_{C_{k+1}\setminus C_k} \gbn\left(\frac{|u-v|}{\vt}\right)dx\right)\\&\leq \frac{\epsilon_2}{\vt}\int_{B_1} \gbn\left(\frac{|u-v|}{\vt}\right)dx+c(s_G,\epsilon_2)\sum_{k=1}^{\infty}H^{-1}\left(\wt{H}\left(\frac{1}{\gbn(k)} \right)\right)\\
&\lesssim \frac{\epsilon_2}{\vt} \int_{B_1} \gb\left( |Du-Dv| \right)dx+c(s_G,\epsilon_2)\sum_{k=1}^{\infty} \frac{1}{\gbn(k)}\\
&\lesssim \frac{\epsilon_2}{\vt} \int_{B_1} \bg\left( |Du-Dv| \right)dx+c(n,s_G,\epsilon_2),\end{split}\]
where we used Young inequality (with arbitrary $\epsilon_2>0$ to be chosen), Sobolev inequality (Theorem~\ref{theo:Sob-emb}), \eqref{tildH-l-H}, $\bg\sim\gb$,~\eqref{H-est}, and, finally, we notice further that the last term is summable due to comparison $\bg$ with power-type functions cf.~\eqref{B-power-compar}. Consequently, we estimate in~\eqref{est-I}
\[\begin{split}&\int_{B_1}\gb(|D\bu-D\bv|)dx \leq I_1+I_2\lesssim (\epsilon_1+\epsilon_2) \int_{B_1}\gb(|D\bu-D\bv|)dx+c(\epsilon_1)+c(n,s_G,\epsilon)\vt\\
&\qquad\qquad\qquad\leq (\epsilon_1+\epsilon_2) \int_{B_1}\gb(|D\bu-D\bv|)dx+c(\epsilon_1)+c(n,s_G,\epsilon)\left({\epsilon_1}  \int_{B_1} \bg (|D\bu-D \bv|)\,dx+c( {\epsilon}_1 )\right)\\
&\qquad\qquad\qquad\leq  \widetilde{\epsilon} \int_{B_1}\gb(|D\bu-D\bv|)dx+\widetilde{c}.\end{split}\]
Choosing $\epsilon_1>0$ and then $\epsilon_2>0$ both sufficiently small, we can absorb the first term on the right-hand side. This fixes the remaining constant and implies~\eqref{eq:resc-comp-est}. 

\bigskip

\textsc{Step 2.2. Fastly growing $G$ }  

\noindent Since $u,v\in W^{1,\bG}(B_1)$, their difference is bounded, thus we can choose
\[\vp=\bu-\bv\in W_0^{1,\bG}(B_1)\cap L^\infty(B_1)\]
as a test function in~\eqref{eq:weak-diff-comp-est}. Notice that Assumption (A) and then  Sobolev's embedding imply
\[\begin{split}\int_{B_1}\bG(|D\bu-D\bv|)\,dx&\lesssim \int_{B_1}\langle \ba(x,D\bu)-\ba(x,D \bv),D\bu-D\bv\rangle\,dx=\int_{B_1}\bu-\bv\, d\bmu\\
& \leq \sup_{B_1}|\bu-\bv| |\bmu|(B_1)\leq c(n,s_G)\|D\bu-D\bv\|_{L_\bG(B_1)},\end{split}
\]
which we estimate further using Lemma~\ref{lem:norm} arriving at
\[\begin{split}\int_{B_1}\bG(|D\bu-D\bv|)\,dx \leq \epsilon \int_{B_1}\bG(|D\bu-D\bv|)\,dx+ \epsilon^{-1}\end{split}
\]
with arbitrary $\epsilon>0$. Let us choose it small enough to make the first term on the right-hand side be absorbed by the left-hand side. Repeating the arguments of~\eqref{gbH} we observe that
\[\begin{split} \int_{B_1} \bg(|D\bu-D\bv|)\,dx\lesssim (\gb\circ \bG^{-1})\left(\barint_{B_1} \bG(|D\bu-D\bv|)dx\right)\leq c
\end{split}\]
and~\eqref{eq:resc-comp-est} is proven. 

\medskip

\noindent\textsc{Step 3. Rescaling back. } We reverse the change of variables from~\eqref{resc-comp}. We have
\[\frac{1}{g(A)}\barint_{B_R}g(|Du-Dv|)dx=\barint_{B_1}\bg(|D\bu-D\bv|)dx\leq c,\]
what completes the proof.\end{proof}

\begin{coro}\label{coro:comparison}
Suppose $u\in W^{1,G}(\Omega)$ is a local SOLA to~\eqref{eq:main}, $v\in u+W^{1,G}_0(B_R)$ is a weak solution to~\eqref{eq:comp-map} on $B_R$, and parameters satisfy $q\in(0,\infty)$, $\theta\in[0,n]$, $\gamma\in(1,\infty)$.

If $\mu(dx)=f(x)dx$ with $f\in L^1(B_R)$, then there exist a constant $c>0$, such that
\begin{equation}
\label{eq:com-f-1st}
 \int_{B_R} g  (|Du-D v|)\,dx \leq c R \int_{B_R} |f|\,dx.
\end{equation}
Moreover, for $\mu\in L^{q,\theta}(B_R)$ there exist a constant $c>0$, such that \begin{equation}
\label{eq:cominMor}\int_{B_R} g  (|Du-D v|)\,dx \leq c R^{\frac{q-\theta}{q}}\|\mu\|_{L^{q,\theta}(B_R)},
\end{equation}
whereas for $\mu\in{L^{\theta}(\gamma,q)(B_R)}$ there exist a constant $c>0$, such that \begin{equation}
\label{eq:cominMorLor}\int_{B_R} g  (|Du-D v|)\,dx \leq c R^{n-\frac{\theta-\gamma}{\gamma}}\|\mu\|_{L^{\theta}(\gamma,q)(B_R)},
\end{equation}
\end{coro}
\begin{proof}Recall that $L^{q,\theta}$ and $L^{\theta}(\gamma,q)$ are defined in Section~\ref{ssec:fn-sp}. Inequality~\eqref{eq:com-f-1st} comes directly from Proposition~\ref{prop:comp-B-xi}  and the norm definition. For~\eqref{eq:cominMor} we apply the H\"older inequality, as for~\eqref{eq:cominMorLor} we apply Proposition~\ref{prop:comp-B-xi} and Lemma~\ref{lem:LqinMarc} with $q=1$ and $t=\gamma$ getting
\[\begin{split}\int_{B_R} g  (|Du-D v|)\,dx &\leq c R^{1+n-\frac{n}{\gamma}}\left(\frac{\gamma}{\gamma-1}\right)\|\mu\|_{{\cal M}^{\gamma}(B_R)}\leq c R^{1+n-\frac{n}{\gamma}}\left(\frac{\gamma}{\gamma-1}\right)\|\mu\|_{L({\gamma},q)(B_R)}\\
&= c R^{n-\frac{\theta-\gamma}{\gamma}}R^{\frac{\theta-n}{\gamma}}\left(\frac{\gamma}{\gamma-1}\right)\|\mu\|_{L({\gamma},q)(B_R)}\leq c R^{n-\frac{\theta-\gamma}{\gamma}}\|\mu\|_{L^{\theta}(\gamma,q)(B_R)}.\end{split}\]
\end{proof}

\section{Preliminary estimates} \label{sec:prel-est}

This section provides preliminary estimates in the scale of the Lorentz the Morrey spaces.

\subsection{Preliminary Lorentz estimates}

We derive estimates on the maximal operator of gradient $Du$ of solutions $u=u_k$ to the problem with bounded data~\eqref{eq:main-k-for-SOLA}.

\subsubsection{Super-level sets estimates}

The important tool we apply is the following version of denisty lemma resulting from~\cite[Lemma~1.2]{CaPe}.
\begin{lem}[Krylov-Safonov density lemma]
\label{lem:covering}
Let $E,F\subset B_0\subset\rn$ be measurable sets with $B_0$ being a ball. Define\begin{equation*}
E^\kappa:=E\cap\kappa B_0\qquad\text{and} \qquad F^\kappa:=F\cap\kappa B_0
\end{equation*}
for every $\kappa\in(0,1]$. Assume that for some ${\delta}\in(0,1)$, $d\geq 1$, and $0<r_1<r_2\leq 1$ the following conditions are satisfied\begin{itemize}
\item $|E^{r_1}|\leq\frac{ {\delta}}{5^n}\left(\frac{r_2-r_1}{d}\right)^n|B_0|$
\item if $B$ is a ball such that $(d B)\subset B_0$, then \[|E\cap B|>\frac{ {\delta}}{5^n}|B|\implies B\subset F.\]
\end{itemize}
Then\begin{equation*}
|E^{r_1}|\leq {\delta}|F^{r_2}|.
\end{equation*}

\end{lem}

Let us prepare to apply the density lemma in consideration on super-level sets of the maximal operator evaluated in gradient of the solution. We employ the notation given by~\eqref{MDu-Mmu-def}. 
\begin{lem}\label{lem:Binc}
Suppose $u\in W^{1,G}(\Omega)$ is a weak solution to~\eqref{eq:main-k-for-SOLA}. Let   $H=H(n,G)>0$ be large fixed absolute constant. Assume further that there exist $T_0$, such that for every $T>T_0$ there exists $\ve=\ve(n,G,T)>0$, such that for every $\lambda>0$ and $B $ such that $2B\subset  B_0$ it holds that
\begin{equation}\label{if}
\Big|B\cap\{x\in B_0: \MDu(x)>HT\lambda\quad \text{ and }\quad\Mmu(x)\leq g(\ve\lambda)\}\Big|>\frac{|B|}{5^n G^{\chi}(HT)}.
\end{equation}
Then
\begin{equation}\label{then}
B\subset \{x\in B_0: \MDu(x)>HT\lambda\}.
\end{equation}
\end{lem}
\begin{proof}
The result is proven by contradiction. We suppose that~\eqref{if} holds, but~\eqref{then} does not. Thus we can consider $\wt{x}\in B$, such that \begin{equation}
\label{M<lambda}
\MDu(\wt{x})\leq \lambda.
\end{equation} Since $4B\subset 2B_0$, also
\begin{equation}
\label{Du<lambda}
\barint_{4B}|Du |dx\leq \lambda.
\end{equation}
Due to~\eqref{if}, for a fixed $\ve>0$ there exists $\bar{x}\in B$, such that $ \Mmu(\bar{x})\leq g(\ve \lambda),$
and consequently \begin{equation}
\label{mu4B}
 {|\mu|(\overline{4B})}{|4 B|^{ \frac{1}{n}-1}}\leq g(\ve\lambda).
\end{equation} 
We use the comparison estimates, between solution to~\eqref{eq:main} and $v\in u+W^{1,G}_0(B_R)$ solving~\eqref{eq:comp-map}. Note that the Jensen inequality   implies
\[\barint_{4B}|Du -Dv |dx= g^{-1}\circ g \left(\barint_{4B}|Du -Dv |dx\right)\leq  g^{-1} \left(\barint_{4B} g \left(|Du -Dv |\right)dx\right),\]
which estimated further due to Proposition~\ref{prop:comp-B-xi}   and~\eqref{mu4B} gives
\begin{equation}
\label{comp-4B-lambda}
\barint_{4B}|Du -Dv |dx\leq c g^{-1}  \left(
 {|\mu|(\overline{4B})}{|4 B|^{ \frac{1}{n}-1}}\right)  \leq c\, \ve\lambda .
\end{equation}
On the other hand, for fixed $\lambda$ function $\bar{v}=v/\lambda$ satisfies rescaled equation of a form~\eqref{eq:comp-map} with the operator $\bar{a}(\cdot)=a(\lambda \cdot)$ controlled by $\bar{G}(\cdot)=  {G\left( {\lambda}\cdot\right) },$  Proposition~\ref{prop:homo-problem}, {\it (ii)} implies that
\begin{equation}\label{comp-map-lambda} 
\begin{split}\barint_{2B} G^{\chi}\left(\frac{ |Dv|}{\lambda}\right)dx
= \barint_{2B}\bar{G}^{\chi}(|Dv|)dx&\leq \bar{c}_1\bar{G}^{\chi}\left(\barint_{4B}|Dv|dx\right)+\bar{c}_2\\
&=\bar{c}_1 {G}^{\chi}\left(\frac{1}{\lambda}\barint_{4B}|Dv|dx\right)+\bar{c}_2\leq  c,\end{split}
\end{equation}
where the last inequality is justified by~\eqref{Du<lambda}.

Then, passing to $B$ and using weak-type estimates, applying $G\in\Delta_2$,~\eqref{comp-map-lambda}, and~\eqref{comp-4B-lambda}, we obtain
\begin{equation}\label{super-level-HTeps}
\begin{split}
&|\{x\in B: \Mb(|Du|)(x)>HT\lambda\}|\\
&\qquad\leq 
|\{x\in 2B: \Mb(|Dv|)(x)>HT\lambda/2\}|+
|\{x\in 2B: \Mb(|Du-Dv|)(x)>HT\lambda/2\}|\\
 &\qquad\leq \frac{c }{G^{\chi}(HT)}\int_{2B}G^{\chi}\left(\frac{2|Dv|}{\lambda}\right)dx+\frac{c }{HT\lambda}\int_{2B}|Du-Dv|dx\\
 &\qquad\leq \frac{cc_{\Delta_2}}{G^{\chi}(HT)}\int_{2B}G^{\chi}\left(\frac{ |Dv|}{\lambda}\right)dx+\frac{c}{HT\lambda}\int_{2B}|Du-Dv|dx\\
 &\qquad\leq c^*|B|\left(\frac{1}{G^{\chi}(HT)}+\frac{\ve}{HT}\right).
\end{split}
\end{equation}
 Note that $c^*>1$. Considering \begin{equation}
\label{choice:eps-H}
\ve= \ve(T,H)=c^{*}\frac{HT}{G^{\chi}(HT) }\qquad\text{and}\qquad H= {20^n}{c^* }>20^n 
\end{equation}
 we obtain in~\eqref{super-level-HTeps}
\begin{equation}\label{B-super-level-T}
\begin{split}
|\{x\in B: \Mb(|Du|)(x)>HT\lambda\}|&\leq  \frac{2|B|}{20^n} \frac{1}{ G^{\chi}(HT) } \leq \frac{|B|}{4\cdot 5^n} \frac{1}{ G^{\chi}(HT) }. 
\end{split}
\end{equation}
To come back to $B_0$ we note that having arbitrary $\tilde{B}\subset 2B$ we can show for any $x\in B$ it holds that 
\begin{equation}
\label{for-B0-super-level-T}
\MDu(x)\leq \max\{\Mb(|Du|)(x),12^n\lambda\}.
\end{equation} considering three cases: $\tilde{B}\subset 2B$, ($\tilde{B}\not\subset 2B$ and $5\tilde{B}\subset 2B_0$), and ($\tilde{B}\not\subset 2B$ and $5\tilde{B}\not\subset 2B_0$).
\begin{itemize}
\item[{\it i)}] When $\tilde{B}\subset 2B$, the definition of the maximal operator enables to write
\[\barint_{\wt{B}}|Du|dy\leq \Mb(|Du|)(x).\]
\item[{\it ii)}] When $\wt{B}\not\subset 2B$ and $5\wt{B}\subset 2B_0$, then there exist $x,y\in\wt{B}$, such that $x\in B$ and $y\not\in 2B$. Let $r_B,r_{\wt{B}},x_B,x_{\wt{B}}$, be such that $B=B(x_B,r_{{B}})$ and ${\wt{B}}=B(x_{\wt{B}},r_{\wt{B}})$. To show that $B\subset 5\tilde{B}$ we fix arbitrary $z\in B$ and notice that
\[|x_{\wt{B}}-z|\leq |x_{\wt{B}}-x|+|x-z|<2r_B+r_{\wt{B}}<5r_{\wt{B}}.\]
Therefore $z\in 5{\wt{B}}$ and in the view of~\eqref{M<lambda} we have 
\[\barint_{\wt{B}}|Du|dy\leq 5^n\barint_{5\wt{B}}|Du|dy\leq 5^n \Mb(|Du|)(\wt{x})\leq 5^n\lambda.\]
\item[{\it iii)}] When $\tilde{B}\not\subset 2B$ and $5\tilde{B}\not\subset 2B_0$, and if $r_{{B_0}},x_{{B_0}}$ are such that $B_0=B(x_{B_0},r_{{B_0}})$, then there exists $z\in B$, such that $|x_{\wt{B}}-z|=5r_{\wt{B}}$ and $|x_{{B_0}}-z|>2r_{{B_0}}.$ For $x\in{\wt{B}}$ we have
\[2r_{B_0}\leq |x_{{B_0}}-z|\leq |x_{{B_0}}-x|+|x-x_{\wt{B}}|+|x_{\wt{B}}-z|\leq r_{B_0}+r_{\wt{B}}+5r_{\wt{B}}=r_{{B_0}}+6r_{\wt{B}}.\]
Consequently, $r_{{B_0}}\leq 6r_{\wt{B}}$ and, due to ~\eqref{M<lambda}, we get 
\[\barint_{\wt{B}}|Du|dy\leq 12^n\barint_{2 {B_0}}|Du|dy\leq 12^n \Mb(|Du|)(\wt{x})\leq 12^n\lambda.\]
\end{itemize}
Thus, we have~\eqref{for-B0-super-level-T}.

Taking into account~\eqref{B-super-level-T} and~\eqref{for-B0-super-level-T}, we get the super-level set estimate for $\MDu$, namely \[|\{x\in B: \MDu(x)>HT\lambda\}|  \leq \frac{|B|}{4\cdot 5^n} \frac{1}{ G^{\chi}(HT) }.\]
This contradicts with~\eqref{if}, which completes the proof.
\end{proof}

Now we are in position to derive  the main tool of the paper, i.e. the super-level set estimates.

\begin{prop}[Super-level set estimates]
\label{prop:super-level-est} Suppose $u\in W^{1,G}(\Omega)$ is a weak solution to~\eqref{eq:main-k-for-SOLA}.  Let $B$ be a ball such that $2B\subset\subset\Omega$ and $0<r_1<r_2\leq 1$. There exist constants $H=H(n,G)>>1$ and $c(n)\geq 1 $, such that the following holds true: for every $T>1$ there exists $\ve=\ve(n,G,T)\in(0,1),$ such that
\begin{equation}
\label{eq:super-level-est}\begin{split}
&|\{x\in r_1 B: \MDu(x)>HT\lambda\}|\\ &\qquad\qquad\leq\frac{1}{G^{\chi}(HT)}|\{x\in r_2 B: \MDu(x)> \lambda\}|+|\{x\in r_1B:\Mmu(x)> g(\ve\lambda)\}|\end{split}
\end{equation}
holds whenever\begin{equation}
\label{lambda>lambda0-def}
\lambda\geq \lambda_0:=\frac{c(n) }{(r_2-r_1)^n}\frac{G^{\chi}(HT)}{HT}\barint_{2B}|Du|\,dx
\end{equation}
\end{prop}
\begin{proof}
Take\[\begin{split}
E&=\{x\in B_0:\ \MDu>HT\lambda\ \ \text{ and }\ \ \Mmu\leq g(\ve\lambda)\},\\
F&=\{x\in B_0:\ \MDu> \lambda\}.
\end{split}\]

Then weak-type estimate implies\[\begin{split}
|E^{r_1}|\leq |E|&\leq |\{x\in B_0:\ \MDu>HT\lambda \}| \leq |\{x\in 2B_0:\ \MDu>HT\lambda \}|\\
&\leq \frac{c2^n|B_0|}{HT\lambda} \barint_{2B_0}|Du|dx.
\end{split}\]
Considering $H$ as in~\eqref{choice:eps-H} and $\lambda\geq \lambda_0$, we infer further
\[|E^{r_1}|\leq  \frac{ 2^n|B_0|(r_2-r_1)^n HT}{20^n HTG^{\chi}(HT) } = \frac{|B_0|}{5^n}\left(\frac{  r_2-r_1}{2}\right)^n\frac{1}{G^{\chi}(HT)}.\]
Therefore, Lemma~\ref{lem:covering} with $E,F$ as above, $ {\delta}=1/G^{\chi}(HT)$, and $d=2$ gives $|E^{r_1}|\leq {\delta}|F^{r_2}|,$
implying~\eqref{eq:super-level-est}.
\end{proof} 

\subsubsection{Lorentz estimates}\label{ssec:Lor-max}

This section is devoted to the proof of the final version of the main tool of the paper, i.e. Lorentz estimates. Recall that the notation given by~\eqref{MDu-Mmu-def}.

\begin{prop} \label{prop:maximal-est}
Suppose $u\in W^{1,G}(\Omega)$ is a weak solution to~\eqref{eq:main-k-for-SOLA} and $\chi$ is the~higher integrability exponent (see Proposition~\ref{prop:homo-problem}, {\it (ii)}). Then for every $(t,\gamma)\in[1, \chi i_G/(s_G-1))\times(0,\infty]$,
there exists a~constant $c=c(c,G,t,\gamma)$ for which
\begin{equation}
\label{grad-est-Lor} 
\avenorm{ g(|Du|)}_{L(t ,\gamma  )(B/2)}\leq c\,g\left( \barint_{2B }|Du|dx\right) +c \avenorm{ \Mbm(\mu) } _{L(t ,\gamma )(B)}
\end{equation}
holds for every $B$, such that $2B\subset\subset \Omega$.
\end{prop}

\begin{proof} We will show Lorentz estimates for the maximal operator 
\begin{equation}
\label{grad-est-1} \avenorm{ g( \Mb(|Du|)) }_{L(t ,\gamma  )(B/2)}\leq c\,g\left( \barint_{2B }|Du|dx\right) +c \avenorm{ \Mbm(\mu) }_{L(t ,\gamma )(B)},
\end{equation}
which directly implies~\eqref{grad-est-Lor} via the Lebesgue differentiation theorem. First we concentrate on the case $\gamma<\infty$ and then $\gamma=\infty$.

\medskip

\textbf{Case $0<\gamma<\infty$.} When we raise both sides of~\eqref{eq:super-level-est} to power $\frac{\gamma}{t}$, then multiply by $g^{\gamma}(HT\lambda)/\lambda$ and integrate, we get 
\begin{multline*} \int_{\lambda_0}^{\lambda_1}g^{\gamma }(HT\lambda) |\{x\in r_1 B: g(\Mb(|Du|)(x))>g(HT\lambda)\}|^\frac{\gamma}{t}\frac{d\lambda}{\lambda}\\
   \leq \frac{c}{G^{\chi \gamma/t}(HT)}\int_{\lambda_0}^{\lambda_1}g^{\gamma }(HT\lambda) |\{x\in r_2 B:  \Mb(|Du|)(x)> \lambda\}|^\frac{\gamma}{t}\frac{d\lambda}{\lambda}\\
   + c\int_{\lambda_0}^{\lambda_1}g^{\gamma }(HT\lambda)  |\{x\in r_1B: \Mbm(\mu)(x) >  g(\ve\lambda) \}|^\frac{\gamma}{t}\frac{d\lambda}{\lambda}.
\end{multline*}
Therefore, changing variables and Lemma~\ref{g<lambda} imply
\begin{multline*}  \int_{g(HT\lambda_0)}^{g(HT\lambda_1)}\lambda^{\gamma} |\{x\in r_1 B: g(\Mb(|Du|)(x))> \lambda\}|^\frac{\gamma}{t}\frac{d\lambda}{\lambda}\\
   \leq  {c }\left(\frac{ (HT)^{s_G-1} }{G^{\chi/t}(HT)}\right)^{\gamma}\int_{\lambda_0}^{\lambda_1}g^\gamma(\lambda)|\{x\in r_2 B: g( \Mb(|Du|)(x))> g(\lambda)\}|^\frac{\gamma}{t}\frac{d\lambda}{\lambda}\\
   + c \left(\frac{ HT }{\ve}\right)^{\gamma(s_G-1)}\int_{\lambda_0}^{\lambda_1}g^\gamma(\ve\lambda) |\{x\in r_1B: \Mbm(\mu)(x) >  g(\ve\lambda) \}|^\frac{\gamma}{t}\frac{d\lambda}{\lambda}\\
   \leq  {c }\left(\frac{ (HT)^{s_G-1} }{G^{\chi/t}(HT)}\right)^{\gamma}\int_{g(\lambda_0)}^{g(\lambda_1)} \lambda^\gamma|\{x\in r_2 B:  g( \Mb(|Du|)(x))>  \lambda \}|^\frac{\gamma}{t}\frac{d\lambda}{\lambda}\\
   + c \left(\frac{ HT }{\ve }\right)^{\gamma(s_G-1)}\int_{g(\lambda_0)}^{g(\lambda_1)}\lambda^{\gamma  }|\{x\in r_1B: \Mbm(\mu)(x) >  \lambda \}|^\frac{\gamma}{t}\frac{d\lambda}{\lambda}.
\end{multline*}

We add to both sides the quantity  
\begin{multline*} \int_0^{g(HT\lambda_0)} \lambda^{\gamma }|\{x\in r_1 B: g(\Mb(|Du|)(x))> \lambda\}|^\frac{\gamma}{t}\frac{d\lambda}{\lambda} \leq \frac{ g^\gamma(HT\lambda_0) }{\gamma +1}|B|^\frac{\gamma}{t}\\
 \leq c(T) |B|^\frac{\gamma}{t}\left(\frac{1}{ r_2-r_1  } \right)^{n \gamma  (s_G-1)}g^\gamma\left(\barint_{2B }|Du|dx\right) ,
\end{multline*}
estimated in the above way due to definition of $\lambda_0$~\eqref{lambda>lambda0-def}. We extend the domain of integration on the right-hand side to obtain
\begin{multline} \label{est:0-HTL} \int_0^{g(HT\lambda_1)}\lambda^{\gamma }|\{x\in r_1 B: g(\Mb(|Du|)(x))> \lambda\}|^\frac{\gamma}{t}\frac{d\lambda}{\lambda}\\
   \leq c(T) |B|^\frac{\gamma}{t}\left(\frac{1}{ r_2-r_1  } \right)^{n \gamma  (s_G-1)}g^\gamma\left(\barint_{2B }|Du|dx\right)\\
   +c\left(\frac{ (HT)^{s_G-1} }{G^{\chi/t}(HT)}\right)^\gamma\int_{0}^{g(HT\lambda_1)}\lambda^{\gamma   }|\{x\in r_2 B: g(\Mb(|Du|)(x))> \lambda\}|^\frac{\gamma}{t}\frac{d\lambda}{\lambda}\\
   + c \left(\frac{ HT}{\ve }\right)^{\gamma(s_G-1)}\int_{0}^{\infty}\lambda^{\gamma}|\{x\in r_1B: \Mbm(\mu)(x) >  \lambda \}|^\frac{\gamma}{t}\frac{d\lambda}{\lambda}
\end{multline}
Furthermore, recall that $H=H(n,G)>0$ is a fixed (big) constant. We can choose $T_0=T_0(n,G)$ big enough for $T>T_0$ to satisfy \[c\left(\frac{  ( HT)^{s_G-1} }{G^{\chi/t}(HT)}\right)^{\gamma}\leq\frac{1}{2},\]
provided $t<\chi i_G/(s_G-1)$. Indeed, note that for large values of $r$ we have \[\left( r^{s_G-1}  G^{-\chi /t}(r)\right)'\leq\frac{1}{t}\Big( t(s_G-1)-\chi {i_G} \Big) {g(r)G^{-\chi/t}(r)}{r^{s_G-2}}.\] As for the last term of~\eqref{est:0-HTL} we observe that
\begin{equation*}\begin{split}
\int_0^{\infty}\lambda^{\gamma }|\{x\in  B:  \Mbm(\mu)(x) > \lambda\}|^\frac{\gamma }{t }\frac{d\lambda}{\lambda} 
&=\frac{1}{\gamma }\| \Mbm(\mu) \|^{\gamma }_{L(t  ,\gamma  )(B)}.\end{split}
\end{equation*}

We apply Lemma~\ref{lem:absorb1} with $R=1$ and
\[\phi(\kappa)= \int_0^{g(HT\lambda_1)}\lambda^{\gamma }|\{x\in \kappa B: g(\Mb(|Du|)(x))> \lambda\}|^\frac{\gamma  }{t  }\frac{d\lambda}{\lambda}\qquad\text{for}\qquad \kappa\in(0,1].\]
To do it we sum up the above remarks, so that the estimate~\eqref{est:0-HTL} becomes 
\[\phi(r_1)\leq \frac{1}{2}\phi(r_2)+c(T) |B|^\frac{\gamma}{t}\left(\frac{1}{ r_2-r_1  } \right)^{n \gamma  (s_G-1)}g^\gamma\left(\barint_{2B }|Du|dx\right)+\frac{1}{\gamma }\|\Mbm(\mu)  \|^{\gamma }_{L(t,\gamma )(B)} \]
and Lemma~\ref{lem:absorb1} implies
\[\int_0^{g(HT\lambda_1)}\lambda^{\gamma }|\{x\in  B/2: g(\Mb(|Du|)(x))> \lambda\}|^\frac{\gamma  }{t  }\frac{d\lambda}{\lambda}\leq c|B|^\frac{\gamma}{t} g^\gamma\left(\barint_{2B }|Du|dx\right) +c \| \Mbm(\mu) \|^{\gamma  }_{L(t  ,\gamma  )(B)}.\]
Now we let $\lambda_1\to\infty$ and get~\eqref{grad-est-1} for $0<\gamma<\infty$.

\bigskip

\textbf{Case $\gamma=\infty$.} Multiplying both sides of~\eqref{eq:super-level-est} by $g^t(HT\lambda)$ and computing supremum  we get
\begin{multline*}
\sup_{\lambda_0\leq\lambda\leq \lambda_1} g^t(HT\lambda) |\{x\in r_1 B: g(\Mb(|Du|)(x))>g(HT\lambda)\}|\\ \leq\frac{ (HT)^{t(s_G-1)} }{G^{\chi}(HT)}\sup_{\lambda_0\leq\lambda\leq \lambda_1} g^t(\lambda)|\{x\in r_2 B: g(\Mb(|Du|)(x))> g(\lambda)\}|\\ +\left(\frac{HT}{\ve}\right)^{t(s_G-1)} \sup_{\lambda_0\leq\lambda\leq \lambda_1} g^t(\ve\lambda)|\{x\in r_1B: \Mbm(\mu)(x) > g(\ve\lambda)\}|.\end{multline*}
As in the case of finite $\gamma$ we change the variables, use the definition of $\lambda_0$, and add  to both sides the initial term
\[\sup_{0\leq\lambda\leq g(HT\lambda_0)}\lambda^{t}|\{x\in r_2 B: g(\Mb(|Du|)(x))>\lambda\}|\leq g^t(HT\lambda_0) |B|\leq \frac{c|B| }{ \left(r_2-r_1  \right)^{n t }} g^t\left(\barint_{2B }|Du|dx\right) \]
to obtain
\begin{equation}
\begin{split}\label{to-est-inf}
&\sup_{ 0\leq \lambda\leq g(HT\lambda_1)}\lambda^{t}|\{x\in r_1 B: g(\Mb(|Du|)(x))> \lambda\}| 
 \leq  \\
&\qquad \qquad\qquad\qquad\qquad  \frac{c|B| }{ \left(r_2-r_1  \right)^{n t }}\  g^t\left(\barint_{2B }|Du|dx\right) \\
&\quad\qquad\qquad\qquad\qquad+\frac{(HT)^{t (s_G-1)}}{G^{\chi}(HT)}\sup_{ 0\leq\lambda\leq g(HT\lambda_1)}\lambda^{t  }|\{x\in r_2 B: g(\Mb(|Du|)(x))> \lambda\}|\\
&\quad\qquad\qquad\qquad\qquad+\left(\frac{HT}{\ve}\right)^{t }\sup_{ 0\leq\lambda\leq \lambda_1}\lambda^{t }|\{x\in r_1B: \Mbm(\mu)(x) > \lambda\}|.\end{split}
\end{equation}
We apply Lemma~\ref{lem:absorb1} with $R=1$ and
\[\phi(\kappa)= \sup_{0\leq \lambda\leq HT\lambda_1}\lambda^{t }|\{x\in \kappa B: g(\Mb(|Du|)(x))> \lambda\}| \qquad\text{for}\qquad \kappa\in(0,1].\]
Note that due to the upper bound on $t$, we can choose $T_0$, such that\[\frac{(HT)^{t(s_G-1) }}{G^{\chi}(HT)}\leq\frac{1}{2}.\] Recall that $H$ is an absolute constant and\[
\sup_{\lambda>0}\lambda^t |\{x\in r_1 B: \Mb(\mu)(x)> \lambda\}|=\|\Mbm(\mu)\|^t_{L(t,\infty)(r_1 B)} \]
from~\eqref{to-est-inf} we get 
\[\phi(r_1)\leq\frac{1}{2}\phi(r_2)+\frac{c|B| }{ \left(r_2-r_1  \right)^{n t }} g^t\left(\barint_{2B }|Du|dx\right) +\|\Mbm(\mu)\|^t_{L(t,\infty)(r_1 B)}.\]
Therefore, Lemma~\ref{lem:absorb1} implies that
\[\sup_{0\leq \lambda\leq HT\lambda_1}\frac{\lambda^{t }}{|B/2|}|\{x\in   B/2: g(\Mb(|Du|)(x))> \lambda\}|\leq c g^t\left(\barint_{2B }|Du|dx\right)+c\| \Mbm(\mu)\|^t_{L(t,\infty)(B)}.\]
To conlude~\eqref{grad-est-1} in the case of $\gamma=\infty$ it suffices to let $\lambda_1\to\infty$.
\end{proof}

\subsection{Preliminary Morrey estimates}

\begin{prop}[Preliminary Morrey estimates]\label{prop:grad-u-est}
Suppose $u\in W^{1,G}(\Omega)$ is a weak solution to~\eqref{eq:main-k-for-SOLA} and $q$ and $\theta$ satisfy~\eqref{q-Morrey}, then there exist a constant $c=c(n,\nu,s_G)>0$, such that
\begin{equation}\label{eq:Morrey-1st-est}
 [ g(|Du|)]_{L^{1,\frac{\theta-q}{q}}(\Omega_1)}\leq c\big({\rm dist}(\Omega_1,\pa \Omega_2)\big)^{ \frac{\theta-q}{q }-n} \|g(|Du|)\|_{L^1(\Omega_2)}+c\|\mu\|_{L^{q,\theta}(\Omega_2)}. \end{equation}
\end{prop}

In fact we will show the result in the broader range of parametrs than~\eqref{q-Morrey}. Namely, the above estimate holds provided
\begin{equation}
\label{wt-q-Morrey}i_G\leq \theta\leq n\qquad\text{and}\qquad 1<q<\frac{\theta i_G\wt{\chi}}{\theta s_G-\theta+i_G\wt{\chi}} 
\end{equation} 
with some $\wt{\chi}=\wt{\chi}(n,i_G,s_G,\nu,L)>1$.

\begin{proof} 

When  $v\in u+W^{1,G}_0(B_R)$ is a solution to~\eqref{eq:comp-map} on $B_R$, we have 
\begin{equation*}
\begin{split}
\int_{B_\vr}g(|Du|)\,dx&\leq 
c\int_{B_\vr}g(|Du-Dv|)\,dx+
c\int_{B_\vr}g(|Dv|)\,dx\\
&\leq 
c\int_{B_R}g(|Du-Dv|)\,dx+c
\left(\frac{\vr}{R}\right)^{n-\beta}\int_{B_R}g(|Dv|)\,dx\\
&\leq 
c\int_{B_R}g(|Du-Dv|)\,dx+
c\left(\frac{\vr}{R}\right)^{n-\beta}\int_{B_R}g(|Du|)\,dx.
\end{split}\end{equation*}
We use above the Jensen inequality, extend the domain of the integration, apply Proposition~\ref{prop:homo-problem} {\it (ii)}, and the fact that $v$ is the solution to the homogeneous problem and thus a minimiser to the~variational formulation.

Let us denote \begin{equation}
\label{gamma:delta}
\wt{\chi}=\min\{\chi,1/\beta\},\qquad\gamma=n-\frac{\theta-q}{q},\qquad\text{and}\qquad \delta=n-\beta(s_G-1),
\end{equation} where $\chi$ is the higher integrability exponent coming from Proposition~\ref{prop:homo-problem}  {\it (ii)} and $\beta$ comes from Proposition~\ref{prop:homo-problem}  {\it (iv)}. Then $\gamma<\delta$. Indeed, due to~\eqref{wt-q-Morrey} we have \[\frac{\wt{\chi}}{\theta}\leq\frac{1}{i_G \beta} \quad\text{and}\quad q<\frac{\theta i_G\wt{\chi}}{\theta s_G-\theta+i_G\wt{\chi}}\leq \frac{\theta}{(s_G-1)\beta+1}.\]
Thus, we can apply Lemma~\ref{lem:absorb2} with \[\phi(\vr)=
\int_{B_\vr}g(|Du|)\,dx,\qquad {\cal B}=\frac{1}{R^{\gamma}}\int_{B_R}g(|Du-Dv|)\,dx,\] $0<\gamma<\delta$ as above  to get
\[\int_{B_\vr}g(|Du|)\,dx\leq c\left(\frac{ \vr}{R}\right)^{\gamma}\left\{ \int_{B_R}g(|Du|)\,dx+  
\int_{B_R}g(|Du-Dv|)\,dx\right\}\quad\text{for }\ \vr\leq R\leq \bar{R}.\]
Therefore, by recalling $\gamma$ from~\eqref{gamma:delta} we end with
\[\vr^{\frac{\theta-q}{q }-n}\int_{B_\vr}g(|Du|)\,dx\leq cR^{ \frac{\theta-q}{q }-n}\left\{\|g(|Du|)\|_{L^1(B_R)}+  
\int_{B_R}g(|Du-Dv|)\,dx\right\}\]
and consequently, for every $\Omega_1\subset\subset\Omega_2 \subset\subset\Omega$, taking into account~\eqref{eq:cominMor}, we get~\eqref{eq:Morrey-1st-est}.
\end{proof}

\section{Proofs of main results}\label{sec:proofs}
This section is devoted to presentation of main proofs and then providing a compact discussion. 
\subsection{Proofs}\label{ssec:main-proofs}
\begin{proof}[Proof of Theorem~\ref{theo:Lorentz-est}] Consider $u_k$ solving~\eqref{eq:main-k-for-SOLA}. Note that within our range of parameter $q<n$. We apply Proposition~\ref{prop:maximal-est}. Recall $(t,\gamma)\in[1, i_G\chi/(s_G-1))\times(0,\infty]$ and take $ {t}=\frac{nq}{n-q} $, we get the above inequality with $q  \in \left(1,\frac{n i_G\chi}{ns_G-n+i_G\chi}\right)$. Therefore, due to Lemma~\ref{lem:Riesz} {\it i)} we get the following inequality 
\[\begin{split} \avenorm{ g( |Du_k|)}_{L\left(\frac{n q}{n-q},\gamma\right)(B_{R/4})}&\leq cg\left(\barint_{B_R}|Du_k|dx\right)+c \avenorm{\Mbm(\mu_k) }_{L\left(\frac{n q}{n-q},\gamma \right)(B)}\\
 &\leq cg\left(\barint_{B_R}|Du_k|dx\right)+c\avenorm{\mu 
 }_{L(q,s)(B)}
,\end{split}\] for all ${B_R\subset\subset\Omega}$. We can pass to the limit with $k\to\infty$ according to assumptions on SOLA (Definition~\ref{def:SOLA}). The  proof of~\eqref{eq:Lorentz-est} can be concluded by a standard covering argument.
\end{proof}

\begin{proof}[Proof of Corollary~\ref{coro:LLogL-est}] Consider $u_k$ solving~\eqref{eq:main-k-for-SOLA}.  Let $B\subset\subset\Omega$. Proposition~\ref{prop:maximal-est} applied with $t=\gamma=n/(n - 1)$ (note that $p\leq n$) and then Lemma~\ref{lem:Riesz} {\it ii)} give
\[\begin{split}\avenorm{ g(|Du_k|)}_{L^\frac{n }{n-1}(B_{R/4})} &\leq c\,g\left( \barint_{B_R }|Du_k|dx\right) +c  \avenorm{ \Mbm(\mu_k)}_{L^\frac{n }{n-1}(B_{R/2})}\\&\leq c\,g\left( \barint_{B_R }|Du_k|dx\right) +c  (n)|B_{R/2}|^\frac{1}{n}\avenorm{ \mu
}_{L \log L(B_{R/2})}
,\end{split}\]
where we can pass to the limit with $k\to\infty$ according to assumptions on SOLA (Definition~\ref{def:SOLA}). Then we obtain the local version of~\eqref{eq:LLogL-est}. Its final form can be obtained via standard covering argument.
\end{proof}

\begin{proof}[Proof of Theorem~\ref{theo:Morrey-est}] As in the previous proofs we shall consider $u_k$ solving~\eqref{eq:main-k-for-SOLA} and after getting the estimates pass to the limit with $k\to\infty$ according to assumptions on SOLA (Definition~\ref{def:SOLA}). We skip $k$ in notation.

 We provide first the estimates for a problem defined on a unit ball $B_1\subset\rn$ and then rescale it to obtain the final estimates. Let us consider $\wt{u}$ solving\begin{equation}
\label{eq:tilde}
 -\dv\wt{a}(D\wt{u})=\wt{\mu}\in L^1( B_1).
\end{equation}
Proposition~\ref{prop:maximal-est} with $t=\gamma=\theta q/(\theta-q)$ yields then
\[\avenorm{ g( |D\wt{u}| )}_{L^{\frac{\theta q}{\theta-q}}(B_{1/8})}\leq cg\left(\barint_{B_{1/2}}|D\wt{u}|dx\right)+c\avenorm{ M^*_{1;B_{1/2}}(\wt{\mu})}_{L^{\frac{\theta q}{\theta-q}}(B_{1/2})}.
\]
We estimate the right-hand side above using the Jensen inequality to get
\begin{equation}
\label{est:tilde}\begin{split}\avenorm{ g( |D\wt{u}| )}_{L^{\frac{\theta q}{\theta-q}}(B_{1/8})}&\leq c\barint_{B_{1/2}}g\left(|D\wt{u}| \right)dx+c\avenorm{ M^*_{1;B_{1/2}}(\wt{\mu})}_{L^{\frac{\theta q}{\theta-q},\theta}(B_{1/2})}\\
&\leq c[g\left(|D\wt{u}| \right)]_{L^{1,\frac{\theta-q}{q}}(B_{1})}+c\avenorm{\wt{\mu}}_{L^{q,\theta}(B_{1/2})},\end{split}
\end{equation} 
where the second line can be obtained due to Lemma~\ref{lem:Emb-inq} {\it i)} and Lemma~\ref{lem:Riesz} {\it iii)}.

Going back to the original solution $u$ we consider a ball $B_\vr=B(x_0,\vr)\subset\subset \Omega$ and rescale the problem. For $y\in B_1$ we put
\[\wt{u}(y) :=u(x_0 +\vr y)/\vr,\quad \wt{\mu}(y) := \vr \mu(x_0 +\vr y),\quad\text{and}\quad \wt{a}(y,z)=a( x_0 + \vr y, z).\] Notice that $\wt{u}$ solves~\eqref{eq:tilde} and we have the estimate~\eqref{est:tilde} for it. Using Remark~\ref{rem:Mor-scal} we infer the estimate for $u$
\begin{equation*}
  \avenorm{ g( |D {u}|)}_{L^{\frac{\theta q}{\theta-q}}(B_{\vr/8})} \leq c\left\{[g\left(|D {u}| \right)]_{L^{1,\frac{\theta-q}{q}}(B_{\vr})}+\| {\mu}\|_{L^{q,\theta}(B_{\vr})}\right\}\vr^\frac{\theta-q}{q},
\end{equation*} 
which by standard covering argument and then by Proposition~\ref{prop:grad-u-est} implies
\begin{equation*}
\| g( |D {u}|)\|_{L^{\frac{\theta q}{\theta-q},\theta}(\Omega_1)}  \leq c \|g(|D {u}|)\|_{L^{1}({\Omega_2})}+\| {\mu}\|_{L^{q,\theta}({\Omega_2})} 
\end{equation*} 
with for $\Omega_1\subset\subset\Omega_2\subset\subset\Omega$ and $c=c(n,G,\nu, L,q,\theta, {\rm dist}(\Omega_1 , \pa\Omega_2 ))$. 

In order to conclude we again use the same scaling argument. For $y\in B_1$ we put 
\[ \bar{u}(y) :=u(x_0 +Ry )/R,\quad \bar{\mu}(y) := R \mu(x_0 +R y),\quad\text{and}\quad \bar{a}(y,z)=a( x_0 + R y, z).\]
and we have
\[\| g( |D \bar{u}|)\|_{L^{\frac{\theta q}{\theta-q},\theta}(B_{3/4})}  \leq c \|g(|D \bar{u}|)\|_{L^{1}({B_1})}+\| \bar{\mu}\|_{L^{q,\theta}({B_1})}. \]
Consequently, again by Remark~\ref{rem:Mor-scal}, we have 
\[\| g( |D  {u}|)\|_{L^{\frac{\theta q}{\theta-q},\theta}(B_{3R/4})}  \leq c R^{\frac{\theta-q}{q}-n}\|g(|D  {u}|)\|_{L^{1}(B_{R})}+\|  {\mu}\|_{L^{q,\theta}({B_R})}. \]
Finally, the final estimate is a consequence of application of Remark~\ref{rem:Mor-scal}.
\end{proof}

\begin{proof}[Proof of Theorem~\ref{theo:Lor-Mor}] Starting as in the proof of Theorem~\ref{theo:Lorentz-est}, but with $ {t}=\frac{\theta q}{\theta-q}$ and $\gamma=\frac{\theta s}{\theta-q}$ and apply Lemma~\ref{lem:Riesz} {\it iv)} we get the following inequality 
\[\begin{split} \avenorm{ g( |Du_k|)}_{L\left(\frac{\theta  q}{\theta-q},\frac{\theta s}{\theta -q}\right)(B_{R/4})}&\leq cg\left(\barint_{B_R}|Du_k|dx\right)+c \avenorm{\Mbm(\mu_k) } _{L\left(\frac{\theta q}{\theta-q},\frac{ns}{n-q} \right)(B)}\\
 &\leq cg\left(\barint_{B_R}|Du_k|dx\right)+c\avenorm{\mu 
 }_{L^\theta(q,\gamma)(B)}
,\end{split}\] for all ${B_R\subset\subset\Omega}$. Then we use it instead of~\eqref{est:tilde} in the reasoning of the proof of Theorem~\ref{theo:Morrey-est} to get the claim.
\end{proof}

\subsection{Comments on the proofs}\label{ssec:comments}

Careful inspection of the proofs of Theorems~\ref{theo:Lorentz-est} ~\ref{theo:Morrey-est}, and~\ref{theo:Lor-Mor} indicates that we actually prove  the expected result in the range of admissible $q$ shall being an open set slightly bigger than in the statement. For brevity we formulate  the main claims  under closed-ended condition capturing the interesting end points. In fact, the proof of Lorentz estimates from Theorem~\ref{theo:Lorentz-est} works provided
\begin{equation*}
 1<q<\frac{n i_G {\chi}}{n s_G-n+i_G {\chi}} \qquad\text{and}\qquad
0<s\leq\infty
\end{equation*} 
with $ {\chi}$ being the higher integrability exponent $ {\chi}= {\chi}(n,i_G,s_G,\nu,L)>1$, whereas the proof of Lorentz estimates from Theorem~\ref{theo:Lorentz-est} under the corresponding restrictions~\eqref{wt-q-Morrey} broader.

\appendix
\addcontentsline{toc}{section}{Appendices}

\section{Appendix}

\subsection{Function  spaces}\label{ssec:fn-sp}
In this section we define and present basic properties of several function  spaces, which are taken into account in the paper. In every definition $\Omega\subset\rn$ is assumed to be an open subset. By the local versions of the spaces defined in this section, we mean naturally those where the norm is finite on arbitrary compact subset of $\Omega$.

\begin{defi}[Lorentz and Marcinkiewicz space]\label{def:Lor:sp}
Let $q,\gamma>0$. A measurable map $f : \Omega\to\r^k$, $k\in\n$ belongs to the Lorentz space $L(q,\gamma)(\Omega)$ if and only if
\[\|f\|_{L(q,\gamma)(\Omega)}= \left(q \int_0^\infty
(\lambda^q |\{x\in\Omega: |f (x)|>\lambda\}|)^{\gamma/q}
\frac{d\lambda}{\lambda}\right)^\frac{1}{\gamma}<\infty.\]
The Marcinkiewicz space ${\cal M}^q(\Omega)=
L(q,\infty)(\Omega)$ is defined setting 
\[\|f\|_{{\cal M}^q(\Omega)}=\left(\sup_{\lambda>0}\lambda^q|\{x\in\Omega: |f (x)|>\lambda\}|\right)^\frac{1}{q}.\]
\end{defi}
Let us point out that the Lorentz spaces are intermediate to the Lebesgue spaces in the following sense: for $0<q<t<r\leq\infty$ 
\[L^r=L(r,r)\subset L(t,q)\subset L(t,t)=L^t\subset L(t,r)\subset L(q,q)= L^q.\]
In particular,
\[L^p\subset {\cal M}^p\subset L^{p-\ve}\]
where the inclusions are proper and for the second one consider a function $|x|^{-n/p}$.

\medskip

We shall make use of the following averaged norms
\[\begin{split}\avenorm{f}_{L(q,\gamma)(\Omega)}&= \left(q \int_0^\infty
\left(\lambda^q\frac{ |\{x\in\Omega: |f (x)|>\lambda\}|}{|\Omega|}\right)^{\gamma/q}
\frac{d\lambda}{\lambda}\right)^\frac{1}{\gamma},\\
\avenorm{f}_{{\cal M}^q(\Omega)}&=\left(\sup_{\lambda>0}\lambda^q\frac{|\{x\in\Omega: |f (x)|>\lambda\}|}{|\Omega|}\right)^\frac{1}{q}.\end{split}\]

\begin{defi}[Morrey space]\label{def:Mor:sp} Let $q\geq 1$ and $\theta\in[0,n]$.  A measurable map $f : \Omega\to\r^k$, $k\in\n$ belongs to the Morrey space $L^{q,\theta}(\Omega)$ 
if and only if \[\|f\|_{L^{q,\theta}(\Omega)}:= \sup\limits_{\substack{
B(x_0,R) \subset\rn\\ x_0\in\Omega}} R^{\frac{\theta-n}{q}} \|f\|_{L^q(B_R\cap \Omega)} <\infty.\]
\end{defi}

Combining the integrability and density conditions we consider also the followig spaces.

\begin{defi}[Lorentz-Morrey and Marcinkiewicz-Morrey spaces]\label{def:LorMor:sp} Let $q\geq 1$ and $\theta\in[0,n]$. We say that $f$ belongs to the Lorentz-Morrey space $L^{\theta}(t,q)(\Omega)$
if and only if\[\|f\|_{L^{\theta}(t,q)(\Omega)}:=\sup\limits_{\substack{
B(x_0,R) \subset\rn\\ x_0\in\Omega}}  R^\frac{\theta-n}{t}\|f\|_{L(t,q)(B_R\cap \Omega)}<\infty,\]
and, accordingly, $f$ belongs to the Marcinkiewicz-Morrey space ${\cal M}^{t,\theta}(\Omega)=L^\theta(t,\infty)(\Omega)$ if and only if\[\|f\|_{{\cal M}^{t,\theta}(\Omega)}:=\sup\limits_{\substack{
B(x_0,R) \subset\rn\\ x_0\in\Omega}}  R^\frac{\theta-n}{t}\|f\|_{{\cal M}^{t }(B_R\cap \Omega)}<\infty.\]
\end{defi}
Obviously, we have
\[{\cal M}^{ q,\theta} \subset {\cal M}^q=  {\cal M}^{ q,0} \quad\text{for }\ q>1,\ \theta\in[0,n],\]
and moreover
\[L^{q,\theta}\subset{\cal M}^{q,\theta}\subset L^{t,\theta} \quad\text{for }\ 1\leq t<q,\ \theta\in[0,n].\]
To visualise how this scale is different than the classical Lebesgue setting let us mention that despite $L^{1,0}=L^\infty$, there exist functions from $L^{1,\theta}$ for $\theta$ arbitrarily close to zero, which   do not belong to $L^q$ for any $q > 1$.

It is sometimes more convenient to consider the localized and averaged norms
 \[\begin{split}[f]_{L^{q,\theta}(\Omega)}:= \sup_{
B_R \subset\Omega} R^{\frac{\theta-n}{q}} \|f\|_{L^q(B_R)},&\qquad [f]_{L^{\theta}(t,q)(\Omega)}:=\sup_{B_R \subset \Omega}  R^\frac{\theta-n}{t}\|f\|_{L(t,q)(B_R)},\\ [f]_{{\cal M}^{t,\theta}(\Omega)}&:=\sup_{
B_R \subset\Omega} R^\frac{\theta-n}{t}\|f\|_{{\cal M}^{t }(B_R)},\end{split}\]
\[\begin{split}
\avenorm{f}_{L^{q,\theta}(\Omega)}&:= \sup\limits_{\substack{
B(x_0,R) \subset\rn\\ x_0\in\Omega}} R^{\frac{\theta-n}{q}} \avenorm{f}_{L^q(B_R\cap \Omega)},\\
\avenorm{f}_{L^{\theta}(t,q)(\Omega)}&:=\sup\limits_{\substack{
B(x_0,R) \subset\rn\\ x_0\in\Omega}}  R^\frac{\theta-n}{t}\avenorm{f}_{L(t,q)(B_R\cap \Omega)},\\
\avenorm{f}_{{\cal M}^{t,\theta}(\Omega)}&:=\sup\limits_{\substack{
B(x_0,R) \subset\rn\\ x_0\in\Omega}}  R^\frac{\theta-n}{t}\avenorm{f}_{{\cal M}^{t }(B_R\cap \Omega)}.\end{split}\]

\begin{rem}\label{rem:Mor-scal} Let us consider $f\in L^{q,\theta} (B)$ with $B=B(x_0,R)$ and $\wt{f}(y) := f (x_0 + Ry)$ for $y$ from the unit ball $B_1$, it follows
\[[\wt{f}]_{ L^{q,\theta} (B_1)} = R^{-\theta/q} [f ]_{ L^{q,\theta} (B)}\qquad\text{and}\qquad \|\wt{f}\|_{ L^{q,\theta} (B_1)} = R^{-\theta/q} \|f \|_{ L^{q,\theta} (B)}.\]
\end{rem}
\begin{rem} Let $f\in L^{ q,\theta}(B_R )$ be a map, $q\geq 1$ and $\theta\in[0, n]$, then\[ \|f \|_{L^{q,\theta}(B_{R/2})}\leq 6^\frac{n-\theta}{q}[f]_{L^{q,\theta}(B_{3R/4})}.\]
\end{rem}
\begin{proof} Consider $B_\vr=B(y,\vr)$ such that $B(y,\vr)\cap B_{R/2}\neq \emptyset$. If $B_\vr\subset B_{R/2}$, then
\[\vr^{\theta-n}\int_{B_\vr \cap B_{R/2}} |f|^qdx\leq \vr^{\theta-n}\int_{B_\vr  } |f|^qdx\leq [f]^q_{L^{q,\theta}(B_{3R/4})}.\]
Otherwise, if $B_\vr \not\subset B_{ 3R/4}$, then $\vr\geq R/8$  and we can estimate
\[\vr^{\theta-n}\int_{B_\vr \cap B_{R/2}} |f|^qdx\leq \left(\frac{R}{8}\right)^{\theta-n}\int_{B_{3R/4}} |f|^qdx\leq 6^{n-\theta}[f]^q _{L^{q,\theta}(B_{3R/4})}.\]\end{proof} 

We shall consider data in the Orlicz space $L\log L$, where the modular function $t\mapsto t\log(e+t)$ satisfies $\Delta_2$-condition, but is growing essentially less rapidly than $t^{1+\ve}$ for any $\ve>0$. 

\begin{defi}[$L\log L$-spaces]\label{def:LLogL:sp} Let $\Omega\subset\rn$  be an open subset of finite
measure  and $k\geq 1$. We define the space $L\log  L(\Omega)$ as a subset of integrable functions $f:\Omega\to\r^k$ such that
\[\int_\Omega |f | \log (e + |f |) dx <\infty,\]
endowed with a norm

\[\begin{split}\|f\|_{L \log  L(\Omega)}&=\inf\left\{\lambda>0:\quad\int_\Omega\left|\frac{f}{\lambda}\right|\log\left(e +\left|\frac{f}{\lambda}\right| \right)dx\leq 1\right\}\\
&\simeq \int_\Omega\left| {f} \right|\log\left(e + \frac{|f|}{\barint_{\Omega}|f(y)|dy}  \right)dx<\infty.\end{split}\]
Moreover, we define  $L \log  L^\theta(\Omega)$ as a subset of integrable functions $f:\Omega\to\r^k$ such that
\[\begin{split}\|f\|_{L \log  L^\theta(\Omega)}&=\sup\limits_{\substack{
B(x_0,R) \subset\rn\\ x_0\in\Omega}}  R^{\theta} \|f\|_{L\log L(B_R\cap \Omega)}\\
&\simeq\sup\limits_{\substack{
B(x_0,R) \subset\rn\\ x_0\in\Omega}}  R^{\theta-n}\int_{B_R}\left| {f} \right|\log\left(e + \frac{|f|}{\barint_{B_R}|f(y)|dy}  \right)dx<\infty.\end{split}\]
\end{defi}

 \subsection{Basics}

The classical reference for this section is~\cite{adams-hedberg}, most of the needed estimates can be found in~\cite{min-grad-est}. 
 
 Let us present basic embedding inequalities.
 
 \begin{lem}\label{lem:Emb-inq}
 Let $B\subset \rn$ be a ball, $n\geq 1$, and  $f:\rn\to\R$  is a locally integrable function supported in $B$. If $1<q<\theta\leq n$,  then there exists $c>0$ such that for every it holds that
\[\avenorm{ f}_{L^1(B)}\leq c \|f\|_{L^{1,\frac{\theta-q}{q}}(B)}\leq c [f]_{L^{1,\frac{\theta-q}{q}}(2B)}.\]
 \end{lem}

\begin{lem}[Lemma~6,~\cite{min-grad-est}]\label{lem:LqinMarc}
Suppose $A\subset\rn$ is measurable. If $h\in{\cal M}^t(A)$, then $h\in L^q(A)$ for every $q\in[1,t)$ and
\[\|h\|_{L^q(A)}\leq\left(\frac{t}{t-q}\right)^\frac{1}{q}|A|^{\frac{1}{q}-\frac{1}{t}}\|h\|_{{\cal M}^t(A)}.\]
\end{lem}

We use the following estimates on maximal operators.

\begin{lem}\label{lem:Riesz} 
 Let $B\subset \rn$ be a ball, $n\geq 2$, and  $\mu:\rn\to\R$  is a locally integrable function supported in $B$. See~\eqref{MDu-Mmu-def} for notation.
 \begin{itemize}
 \item[i)]
If $1<q<n$ and $s \in(0,\infty]$, then there exists $c = c(n,q,s)$ such that for every it holds that \[\| \Mbm(\mu)\|_{L\left(\frac{n q}{n-q},s\right)(B)}\leq\ c \|\mu \|_{L(q,s)(B)}.\] 
 \item[ii)]
There exists $c = c(n)$ such that for every it holds that
 \[\| M_{1,B}^*(\mu)\|_{L^\frac{n}{n-1}(B)}\leq c(n)|B|^\frac{1}{n} \| \mu\|_{L \log L(B)}.\]
 \item[iii)] If $1<q<\theta\leq n$,  then there exists $c >0$ such that for every it holds that
\[\| M_{1;B}^*(\mu)\|_{L^{\frac{\theta q}{\theta-q},\theta}(B)}\leq c \|\mu\|_{L^{q,\theta}(B)}.\]
 \item[iv)] If $1<q<\theta\leq n$ and $s\in(0,\infty)$,  then there exists $c >0$ such that for every it holds that
\[\|M_{1;B}^*(\mu)\|_{L\left(\frac{\theta q}{\theta-q},\frac{\theta s}{\theta-q}\right)(B)}\leq c \|\mu\|_{L^{\theta}(q,s)(B)}.\]
\end{itemize} 
\end{lem}

We shall use two different classical absorption lemmas, both to be find in~\cite{Giusti}.
\begin{lem}[\cite{Giusti}, Lemma~6.1]
\label{lem:absorb1}
Let $\phi:[R/2,3R/4]\to[0,\infty)$ be a function such that
\[\phi(r_1)\leq \frac{1}{2}\phi(r_2)+\mathcal{A}+\frac{\mathcal{B}}{(r_2-r_1)^\beta}\qquad\text{for every}\qquad R/2\leq r_1<r_2\leq 3{R}/4\]
with $\mathcal{A,B}\geq 0$ and $\beta>0$. Then there exists $c=c(\beta)$, such that\[\phi(R/2)\leq c \left(\mathcal{A}+\frac{\mathcal{B}}{R^\beta}\right).\]
\end{lem}
\begin{lem}[\cite{Giusti}, Lemma~7.3]
\label{lem:absorb2}
Let $\phi:[0,\bar{R}]\to[0,\infty)$ be a non-decreasing function such that
\[\phi(\vr)\leq c_0\left(\frac{\vr}{R}\right)^{\delta }\phi(R)+ {\mathcal{B}}R^{\gamma}\quad\text{for every }\vr\leq R\leq \bar{R},\]
with some $0<\gamma<\delta $ and ${\mathcal{B}}>0$. Then there exists $c=c(c_0,\gamma)$, such that\[\phi(\vr)\leq c\left\{\left(\frac{\vr}{R}\right)^{\gamma}\phi(R)+  {\mathcal{B}}\vr^{\gamma}\right\}\quad\text{for every }\vr\leq R\leq \bar{R}.\]
\end{lem}

For the following inequality we refer to~\cite{Baroni-Riesz} or \cite{CGZG} without the growth restrictions, where the constant can be moved out of the modular function under $\Delta_2$-condition. 
\begin{lem}[The Sobolev inequality] \label{lem:Sob}
Let $h\in C^1(0,\infty)$ be an increasing and convex function satisfying~\eqref{ig-sg}. Assume further that $B_R\subset\rn$ is a ball.  Then there exists a constant $c=c(n,s_h)$, such that
\[\barint_{B_R} h^{\frac{n}{n-1}}(|u|)dx\leq c \left(\barint_{B_R} h(|Du|)dx\right)^{\frac{n}{n-1}}\]
for every weakly differentiable function $u\in W_0^{1,h}(B_R )$.
\end{lem}

\begin{lem}[cf.~\cite{adams-fournier}]\label{lem:G*} Suppose $h$ is an increasing and convex function and $\wt{h}$ is its Young conjugate, then there exists a constant $c$, such that for every $t>0$ we have
\[\wt{h}\left(h(t)/t\right)\leq ch(t).\]
\end{lem}

\begin{lem}\label{g<lambda} Suppose $H$ is an increasing and convex function $H\in C^1(0,\infty)$ satisfying $\Delta_2$-condition. If $h(t)=H'(t)$, then there exists a constant $c$, such that for every $t>0$ and $\lambda>1$, we have
\[h(\lambda t)\leq s_H \lambda^{s_H-1} h(t).\]
\end{lem}
\begin{proof}Indeed, since $H\in  \Delta_2$, we have $h(t)\leq s_H H(t)/t$. On the other hand, ${H(\lambda t)}/{(\lambda t)^{s_H}}\leq    {H(t)}/{ t ^{s_H}}$ and $\Delta_2$-condition implies that $t\mapsto H(t)/t^{s_H}$ is non-increasing. 
\end{proof}

\bibliographystyle{plain}
\bibliography{ICgradest}

\end{document}